\documentclass{amsart}
\usepackage{amssymb}
\usepackage{graphicx}
\usepackage{color}
\usepackage{subfigure}
\usepackage[all]{xy}
\usepackage{calc}
\usepackage{hyperref}

\newtheorem{thm}{Theorem}[section]
\newtheorem{lem}[thm]{Lemma}
\newtheorem{prop}[thm]{Proposition}
\newtheorem{cor}[thm]{Corollary}
\theoremstyle{definition}
\newtheorem{defn}[thm]{Definition}
\newtheorem{ex}[thm]{Example}
\newtheorem{rmk}[thm]{Remark}

\newcommand\val{\operatorname{val}}
\newcommand\lk{\operatorname{lk}}
\newcommand\st{\operatorname{st}}
\newcommand\br{\operatorname{br}}
\newcommand\coker{\operatorname{coker}}
\newcommand\codim{\operatorname{codim}}
\newcommand\colim{\operatorname{colim}}

\begin{document}
\title{On the structure of braid groups on complexes}
\author{Byung Hee An}
\email{anbyhee@ibs.re.kr}
\address{Center for Geometry and Physics, Institute for Basic Science (IBS), Pohang 37673, Republic of Korea}
\thanks{This work was supported by IBS-R003-D1.}
\author{Hyo Won Park}
\email{park.hyowon@gmail.com}
\address{Department of Mathematics, Ajou University, Suwon 16499, Republic of Korea}
\begin{abstract}
We consider the braid groups $\mathbf{B}_n(X)$ on finite simplicial complexes $X$, which are generalizations of those on both manifolds and graphs that have been studied already by many authors.
We figure out the relationships between geometric decompositions for $X$ and their effects on braid groups, and provide an algorithmic way to compute the group presentations for $\mathbf{B}_n(X)$ with the aid of them.

As applications, we give complete criteria for both the surface embeddability and planarity for $X$, which are the torsion-freeness of the braid group $\mathbf{B}_n(X)$ and its abelianization $H_1(\mathbf{B}_n(X))$, respectively.
\end{abstract}
\maketitle
\tableofcontents

\section{Introduction}

The braid group $\mathbf{B}_n(D^2)$ on a 2-disk $D^2$ was firstly introduced by E.~Artin in 1920's, and Fox and Neuwirth generalized it to braid groups $\mathbf{B}_n(X)$ on arbitrary topological spaces $X$ via {\em configuration spaces}, which are defined as follows. For a compact, connected topological space $X$, the {\em ordered configuration space $F_n(X)$} is the set of $n$-tuples of distinct points in $X$, and the orbit space $B_n(X)$ under the action of the symmetric group $\mathbf{S}_n$ on $F_n(X)$ permutting coordinates is called the {\em unordered configuration space} on $X$.
\[F_n(X)=X^n\setminus \Delta,\quad B_n(X)=F_n(X)/\mathbf{S}_n,\]
where
\[\Delta=\{(x_1,\dots,x_n)|x_i=x_j \text{ for some }i\neq j\}\subset X^n.\]

Let $\bar*_n$ and $*_n$ be basepoints for $F_n(X)$ and $B_n(X)$, respectively.
Then the {\em pure $n$-braid group} $\mathbf{P}_n(X,\bar*_n)$ and {\em (full) $n$-braid group $\mathbf{B}_n(X,*_n)$} are defined to be the fundamental groups of the configuration spaces $F_n(X)$ and $B_n(X)$, respectively. We will suppress basepoints and denote these groups by $\mathbf{P}_n(X)$ and $\mathbf{B}_n(X)$ unless any ambiguity occurs.

However, most of research on braid groups has been focused on braid groups on manifolds, more specifically, on surfaces, until the end of 20th century when Ghrist presented a pioneering paper \cite{Gh} about braid groups on {\em graphs $\Gamma$} which are finite, 1-dimensional simplicial complexes.
In 2000, Abrams defined in his Ph.~D. thesis \cite{Ab} a combinatorial version of configuration space, called a {\em discrete configuration space}, consisting of $n$ open cells in $\Gamma$ having pairwise no common boundaries. 
A discrete configuration space has the benefit that it admits a cubical complex structure making the description of paths of points easier. However it depends not only on homeomoprhic type but also the cell structure of the underlying graph $\Gamma$.
Abrams overcame this problem by proving stability up to homotopy under the subdivision of edges once $\Gamma$ is sufficiently subdivided.

Crisp and Wiest showed the embeddabilities between braid groups on graphs and surface groups into right-angled Artin groups, which is one of the most important subjects in geometric group theory. 
Farley and Sabalka in \cite{FS} used Forman's {\em Discrete Morse theory} \cite{For} on discrete configuration spaces to provide an algorithmic way to compute a presentation of $\mathbf{B}_n(\Gamma)$, and furthermore they figured out the relation between braid groups on trees and right-angled Artin groups.
On the extension of these works, Kim-Ko-Park in \cite{KKP} and Ko-Park in \cite{KP} provided geometric criteria for the braid group on a given graph to be right-angled Artin, and moreover a new algebraic criterion for the planarity of a graph, and answered some open questions as well.

On the contrary, for a simplicial complex, not manifold, of dimension 2 or higher, braid theory is still unexplored. We will focus on the braid groups on finite, connected simplicial complex $X$ of arbitrary dimension, which are generalizations of both graphs and surfaces. 
We consider modifications---attaching or removing higher cells, edge contraction or inverses, and so on--- and how these modifications change the braid groups.
Indeed, via suitable modifications we may obtain a {\em simple} complex $X'$ of dimension 2 whose vertices have very obvious links. Furthermore, this can be done without changing the braid group.

\begin{thm}\label{thm:simple}
Let $X$ be a complex. Then there is a simple complex $X'$ of dimension $2$ such that
$\mathbf{B}_n(X)\simeq\mathbf{B}_n(X')$ for all $n\ge 1$.
\end{thm}

Once we have a simple complex $X$, then it can be decomposed by {\em cuts} into much simpler pieces, and eventually into {\em elementary} complexes, where an elementary complex plays the role of a building block and can be thought as either a {\em star} graph or a manifold of dimension at least 2.
For the build-up process, we provide two types of {\em combination theorem} which are generalizations of capping-off and connected sum.
Furthermore, the combination theorems ensure that the build-up process preserves some geometry of the given pieces. In other words, the braid group $\mathbf{B}_n(X)$ captures some geometric properties of $X$ as observed before.

More precisely, we start with the obvious observations about the various embeddabilities of $X$ into manifolds as follows. 
For two complexes $X$ and $Y$, we denote by $Y\subset X$ and say that $X$ {\em contains} $Y$ if there is an simplicial embedding between them after sufficient subdivisions.
Then a complex $X$ embeds into
(i) a circle iff $T_3\not\subset X$; 
(ii) a surface iff $S_0\not\subset X$; and
(iii) a plane iff $K_5, K_{3,3}, S_0\not\subset X$.

The complexes $T_3$ and $S_0$ are the {\em tripod} and the cone $C(S^1\sqcup\{*\})$ of the union of a circle and a point, respectively. See Figure~\ref{fig:S_0}. The graphs $K_n$ and $K_{m,n}$ are complete and complete bipartite graphs, respectively. 

\begin{figure}[ht]
\[
T_3=\vcenter{\hbox{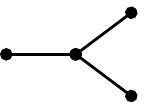}}\qquad\qquad
S_0=\vcenter{\hbox{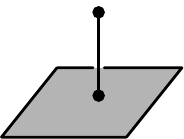}}
\]
\caption{A tripod $T_3$ and a complex $S_0$}
\label{fig:S_0}
\label{fig:tripod}
\end{figure}
 
Then it can be formulated as follows. 
\begin{thm}\label{thm:emb}
Let $X$ be a finite, connected simplicial complex different from $S^2$ and $\mathbb{R}P^2$. Then $X$ embeds into
\begin{enumerate}
\item a circle if and only if $\mathbf{B}_n(X)$ is abelian for any $n\ge 1$;
\item a surface if and only if $\mathbf{B}_n(X)$ is torsion-free for any $n\ge 1$;
\item a plane if and only if $H_1(\mathbf{B}_n(X))$ is torsion-free for any $n\ge 1$.
\end{enumerate}
Moreover, if $X$ does not embed into any surface, then $\mathbf{B}_n(X)$ contains $\mathbf{S}_n$ for any $n\ge 1$.
\end{thm}

Remark that we exclude the cases for $S^2$ and $\mathbb{R}P^2$ since their braid groups have torsion even though they are braid groups on {\em surfaces}, 2-dimensional manifolds.
However, by the complementary statement, they are classified as $\mathbf{B}_n(X)$ contains a torsion but not the whole $\mathbf{S}_n$ for any $n\ge 3$.

The rest of this paper is organized as follows. In Section~2, we define braid groups on complexes and basic notions. In Section~3, we define the modifications and prove Theorem~\ref{thm:simple}, and in Section~4 we define two operations, called {\em unwrapping} and {\em connected-sum decomposition}, and look at the shapes of the elementary complexes.
The effects of the inverses, called {\em closure} and {\em connected sum}, of these two operations on braid groups will be discussed separately in Section~5 and Section~6, which they let us know how the build-up process is working.
Finally, in Section 7, as applications we prove the criteria, Theorem~\ref{thm:emb}, for the embeddability of given complex $X$ into a surface and a plane.

\section{Braid groups on complexes}
Throughout this paper, a {\em complex} denoted by $X$ means a finite, connected, simplicial complex of dimension at least 1.
Especially, a complex of dimension 1 is usually denoted by $\Gamma$ and called a {\em graph}.
Since the braid group on $X$ depends only on the homeomorphism type of $X$, we sometimes assume that $X$ is {\em sufficiently subdivided}, which can be achieved via the barycentric subdivision twice.
The {\em star} $\st(K)$ is the union of all open simplices whose closure intersects $K$, and the {\em link} $\lk(K)$ of $K$ is the complement of the star $\st(\overline K)$ of the closure $\overline K$ in its closure $\overline{\st(K)}$. That is, $\lk(K)=\overline{\st(K)}\setminus \st(\overline{K})$, as usual.

Note that both $F_n(X)$ and $B_n(X)$ can be regarded as finite simplicial complexes up to homotopy as follows.
Since $X$ is a finite simplicial complex, so is the $n$-fold product $X^n$, and after barycentric subdivisions if necessary, the diagonal $\Delta$ becomes a simplicial subcomplex of $X^n$.
Hence the further subdivision makes $X^n\setminus \st(\Delta)$ a strong deformation retract of $X^n\setminus \Delta=F_n(X)$. Therefore if we endow a metric $d$ on $X$, we may assume that there exists a constant $\epsilon=\epsilon(X)>0$ such that any two points of $\mathbf{x}\in F_n(X)$ never approach within $\epsilon$ of each other with respect to the metric $d$, and the same holds for $B_n(X)$.

From the definitions of configuration spaces, we have the following exact sequence.
\begin{equation}\label{eq:bs}
1\longrightarrow \mathbf{P}_n(X,\bar*_n)\longrightarrow \mathbf{B}_n(X,*_n)\stackrel{\rho}{\longrightarrow}\mathbf{S}(*_n)
\end{equation}
Here the group $\mathbf{S}(*_n)$ is the symmetric group on the set $*_n$, usually denoted by $\mathbf{S}_n$, and the map $\rho$ is called the {\em induced permutation}.
It is easy to see that $\mathbf{P}_n(I)=\mathbf{B}_n(I)=\{e\}$, and $\mathbf{P}_n(S^1)=n\mathbb{Z}\subset\mathbb{Z}=\mathbf{B}_n(S^1)$.

On the other hand, it is known that $\rho$ for $\mathbf{B}_n(T_3)$ on a tripod $T_3$ is surjective for each $n\ge 2$. Hence whenever $X$ contains $T_3$, then $\rho$ is surjective as well. 

\begin{prop}
Let $X$ be a complex. Then $X$ embeds into a circle if and only if $\mathbf{B}_n(X)$ is abelian for any $n\ge 1$.
\end{prop}
\begin{proof}
The only if part is obvious.

Suppose $\mathbf{B}_n(X)$ is abelian. Then since $\mathbf{S}_n$ is non-abelian for any $n\ge3$, $\rho$ never be surjective. Hence $X$ is either $I$ or $S^1$, and therefore it embeds into a circle.
\end{proof}
 
We call $X$ {\em trivial} if $X$ is either $I$ or $S^1$. Then $T_3$ can be thought as the obstruction complex for given complex to be trivial.
From now on we assume that $X$ is non-trivial.

\begin{defn}
For any $x\in X$, there is a trichotemy as follows.
\begin{enumerate}
\item $x$ is in the {\em interior $\mathring{X}$} if $\lk(x)\simeq S^k$ for some $k\ge 0$;
\item $x$ is in the {\em boundary $\partial X$} if $\lk(x)\simeq D^k$ for some $k\ge 0$;
\item $x$ is in the {\em branch set $\br(X)$} of $X$ otherwise.
\end{enumerate}
\end{defn}

\begin{defn}
Let $X$ be a complex.
\begin{enumerate}
\item A 0-cell $v$ is called a {\em vertex}, whose {\em valency $\val(v)$} is defined by the number of connected components of $\lk(v)$.
\item
A 1-cell $e=(v,w)$ is called an {\em edge} if there is no 2-cell containing $e$ in its boundary.
\item
For a subset $K$ of $X$, a {\em deletion $X_K$ of $K$} in $X$ is defined by the complement $X\setminus \st(K)$ of $\st(K)$. 
\end{enumerate}
\end{defn}

\begin{figure}[ht]
\def\svgwidth{.95 \textwidth}
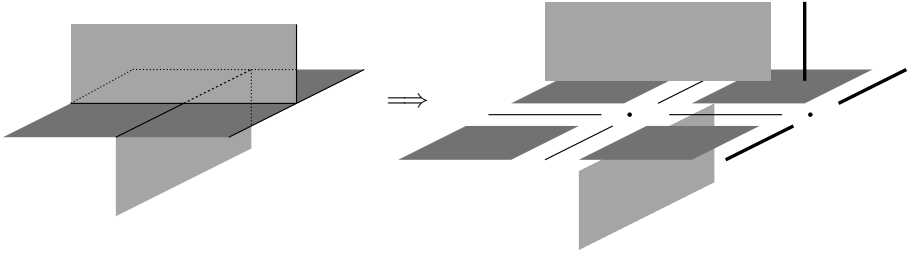
\caption{A decomposition of $X$ into sets of interior, boundary, and branch points}
\label{fig:IntBoundaryBranch}
\end{figure}

The Figure~\ref{fig:IntBoundaryBranch} shows an example.
The thin lines and dots are in $\br(X)$, and the thick lines are in $\partial X$.
Note that $\br(X)$ is a closed subcomplex of $X$, and $X$ is a manifold if and only if $\br(X)=\emptyset$.

\begin{thm}\cite{Bir}
Let $M$ be a manifold of dimension at least $3$, not necessarily compact and possibly with boundary. Then the pure and full braid groups are as follows.
\[\mathbf{P}_n(M)=\prod^n \pi_1(M),\qquad
\mathbf{B}_n(M)=\prod^n\pi_1(M)\rtimes\mathbf{S}_n,\]
where the symmetric group $\mathbf{S}_n$ acts on the product $\prod^n\pi_1(M)$ by permuting factors.
\end{thm}

Hence there is no braid theory for manifolds of dimension 3 or higher.
On the other hand, for a surface $\Sigma$, then there is a fiber bundle structures between the ordered configuration spaces $F_n(\Sigma)$'s which can be used to compute and analyze braid groups on $\Sigma$. 
Note that since compact surfaces are completely characterized by a few parameters, so are their braid groups. 
Indeed, for a given surface $\Sigma$, one can extract geometric information from its braid group as follows.
The proof is obvious by the group presentation for $\mathbf{B}_n(\Sigma)$, see \cite{Bel}, and we omit the proof.

\begin{thm}\cite{Bel, Bir, GG}\label{thm:surface}
Let $\Sigma$ be a surface. Then the following holds.
\begin{enumerate}
\item $\mathbf{B}_n(\Sigma)$ has torsion if and only if $\Sigma$ is either $S^2$ or $\mathbb{R}P^2$.
\item The abelianization $H_1(\mathbf{B}_n(\Sigma))$ has torsion if and only if $\Sigma$ is nonplanar.
\end{enumerate}
\end{thm}

On the contrary, if $X$ is not a manifold, the global topology for a complex $X$ can hardly be determined by a few parameters in general, even if $X$ is 1-dimensional.
Therefore one might not expect that similar results hold for $X$, but surprisingly, the braid group still detects some of the global geometry of the complex $X$ when $X$ is a graph $\Gamma$ as follows.
\begin{thm}\cite{KKP, Gh}
Let $\Gamma$ be a graph. Then $\mathbf{B}_n(\Gamma)$ is {\em always} torsion free, and moreover $H_1(\mathbf{B}_n(\Gamma))$ has torsion if and only if $\Gamma$ is nonplanar.
\end{thm}

Hence Theorem~\ref{thm:emb} is the generalization of the two theorems above, and to prove our theorem, we adopt the notion from graph theory which is the mimic of the {\em minor} relation, that is, edge contraction and deletion. Note that the minor relations reduce the number of edges and so the result is usually considered as simpler than the original one. However, they may increase valencies, and the higher valency tends to imply a more complicated situation in the computation of the braid group. Therefore we may think that a complex having lower valencies is simpler.

Rigorously speaking, we define a simple complex as follows.
\begin{defn}
A vertex $v$ is said to be {\em simple} if $\lk(v)$ is either connected, a disjoint union of a connected complex and a point, or 0-dimensional.
A complex $X$ is said to be {\em simple} if all vertices in $X$ are simple.
\end{defn}

Hence a simple complex is really easy to handle, but is too special and far from the generic ones.
However, we claim that any complex can be transformed into a simple complex by a sequence of certain modifications such as attaching and removing (higher) cells, where each step induces an isomorphism between braid groups.

For convenience's sake, we say that an embedding $f:X\to Y$ is a {\em braid equivalence} if it induces an isomorphism $f_*:\mathbf{B}_n(X)\to\mathbf{B}_n(Y)$ for each $n\ge 1$, respectively.
Moreover, we simply say that $X$ and $Y$ are {\em braid equivalent} if they can be joined by a sequence of (possibly {\em inverse} of) braid equivalences, denoted by $X\equiv_B Y$, respectively.

Then the above claim can be reformulated as for any $X$, there is a simple representative in the braid equivalence class of $X$ as presented in Theorem~\ref{thm:simple}. We will prove this proposition later.

\subsection{Appending a point}
Let $v\in\partial X$, and $i_v:B_{n-1}(X\setminus\{v\})\to B_n(X)$ for $n\ge 1$
be an embedding defined as
\[i_v(\mathbf{x})=\{v\}\cup \mathbf{x}\]
for $\mathbf{x}\in B_{n-1}(X\setminus\{v\})$. 
Note that $i_v(\emptyset) = \{v\}$ if $n=1$.

Then it induces a homomorphism 
\[(i_v)_*:\mathbf{B}_{n-1}(X\setminus \{v\},*_{n-1})\to\mathbf{B}_n(X,i_v(*_{n-1})),\]
where $*_{n-1}$ is a basepoint for $B_{n-1}(X\setminus\{v\})$.

Note that $B_n(X\setminus \{v\})$ is homotopy equivalent to $B_n(X)$ via the inclusion, whose homotopy inverse is a map $h$ resizing cells incident to $v$.
Hence we can consider a composition
\[\bar i_v:B_{n-1}(X)\stackrel{h}\longrightarrow B_{n-1}(X\setminus\{v\})\stackrel{i_v}\longrightarrow B_n(X),\]
which induces
\[(\bar i_v)_*:\mathbf{B}_{n-1}(X,*_{n-1})\to\mathbf{B}_n(X,\bar i_v(*_{n-1})).\]
We can use safely $i_v$ and $(i_v)_*$ instead of $\bar i_v$ and $(\bar i_v)_*$ since there is no ambiguity up to homotopy.

\begin{prop}\label{prop:injection}
Let $X$ be a complex and $v\in\partial X$.
Then the homomorphism $(i_v)_*:\mathbf{B}_{n-1}(X)\to\mathbf{B}_n(X)$ is injective for all $n\ge1$.
\end{prop}
\begin{proof}
Let $B_{1,n-1}(X)=F_n(X)/\mathbf{S}_{n-1}$ by considering $\mathbf{S}_{n-1}$ as a subgroup of $\mathbf{S}_n$ which consists of permutations on $\{1,\dots,n\}$ fixing 1. Then
\[B_{1,n-1}(X)=\{(x_1,\{x_2,\dots,x_n\})| x_i\neq x_j\text{ if }i\neq j\},\]
and the quotient map $p:B_{1,n-1}(X)\to B_n(X)$ forgetting the order is a covering map which is non-regular in general.
Moreover, $i_v$ lifts to $\tilde i_v:B_{n-1}(X\setminus v)\to B_{1,n-1}(X)$ defined by $\tilde i_v(\mathbf{x}) = (v,\mathbf{x})$ and there is a map $\pi:B_{1,n-1}(X)\to B_{n-1}(X)$ forgetting the first coordinate, namely,
\[\pi( x_1, \{x_2,\dots,x_n\}) = \{x_2,\dots,x_n\},\]
satisfying that $\pi\circ\tilde i_v$ is homotopic to the identity, and it induces the isomorphism $\pi_*\circ (\tilde i_v)_*$.
Hence $(\tilde i_v)_*$ is injective and therefore so is $(i_v)_* = p_*\circ(\tilde i_v)_*$.
\end{proof}

\section{Modifications and simple complexes}
\subsection{Edge contraction}
The first nontrivial observation is about edge contraction as follows.
\begin{prop}\label{prop:generaledgecontraction}
Let $X$ be a complex and $e$ be an edge. Then the quotient map $q:X\to X/\bar e$ induces a map 
\[q^*:\mathbf{B}_n(X/\bar e)\to\mathbf{B}_n(X),\]
which is surjective if none of $\partial e$ is of valency 1. 
\end{prop}
\begin{proof}
We first consider a subspace $B_{n;\le 1}(X;\bar e)$ of $B_n(X)$ consisting of configurations $\mathbf{x}=\{x_1,\dots,x_n\}$ such that at most 1 of $x_i$'s is lying on $\bar e$, or equivalently, for $\mathbf{x}\in B_n(X)$,
\[\mathbf{x}\in B_{n;\le 1}(X;\bar e)\Longleftrightarrow \#(\mathbf{x}\cap \bar e)\le 1.\]
Then the map $q$ induces $q|:B_{n;\le1}(X;\bar e)\to B_n(X/\bar e)$.

Let $*_n\subset X\setminus \bar e$ be a basepoint for both $B_{n;\le 1}(X;\bar e)$ and $B_n(X/\bar e)$, and
let a path $\gamma:(I,\partial I)\to (B_n(X/\bar e), *_n)$ be given.
Then it is not hard to prove that there exists a lift $\tilde\gamma:(I,\partial I)\to B_{n;\le 1}(X;\bar e)$ so that $q\circ\tilde\gamma = \gamma$ by regarding $\bar e$ as a path. Moreover, the lift is unique up to homotopy since $\bar e$ is contractible.
Therefore, $q|$ induces an isomorphism $(q|)_*$ between fundamental groups. Then the map $q^*$ is defined by a composition $\iota_*\circ (q|)_*^{-1}$, where $\iota_*$ is the map induced from the obvious inclusion $\iota:B_{n;\le 1}(X;\bar e)\to B_n(X)$, and is well-defined as desired.

\begin{figure}[ht]
\[
\xymatrix{
X=\vcenter{\hbox{\def\svgscale{1.2}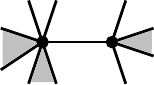}}\ar[r]^-q&
\vcenter{\hbox{\def\svgscale{1.2}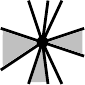}}=X/\bar e\\
\tilde\gamma=\vcenter{\hbox{\def\svgscale{1.2}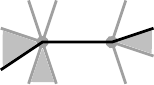}}& 
\vcenter{\hbox{\def\svgscale{1.2}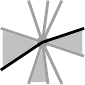}}=\gamma\ar@{|->}[l]
}
\]
\caption{Local pictures of $X/\bar e$ and $X$, and the lift $\tilde \gamma$ of a path $\gamma$}
\label{fig:Xe}
\end{figure}

Suppose that none of $\partial e$ is of valency 1.
Then for the surjectivity, it is enough to show that $\iota_*$ is surjective. In other words, any $\delta:(I,\partial I)\to(B_n(X),*_n)$ is homotoped to $\delta'$ relative to the boundary such that $\#(\delta'(t)\cap \bar e)\le 1$ for all $t\in[0,1]$. 

At first break $\delta$ into several pieces according to the change of $m(t)=\#(\delta(t)\cap \bar e)$, and use induction on $m$.
Then since both $\val(v)$ and $\val(w)\ge 2$, we have enough room for a given configuration to be evacuated from $e$.
This can be done easily and we omit the detail. 
\end{proof}

If none of $\partial e$ is of valency 1, then we may say that $X$ is simpler than $X/\bar e$ according to the definition of simplicity.
Note that $q$ does not directly induce the map between braid groups since it is not an embedding.
Under some conditions, one can find an embedding which plays a similar role to $q$ so that it induces precisely the inverse of $q^*$, and therefore an isomorphism. We will see this later. 

On the other hand, if one of $\partial e$ is of valency 1, then $q$ can be considered as a strong deformation retract and therefore $q^*$ is actually induced from the obvious embedding $X/\bar e\to X$ which is a homotopy inverse of $q$.
However, $q^*$ is neither injective nor surjective in general. It depends on the structure of $\st(e)$.

\begin{ex}[2-braid group on a tree]\label{ex:tree}
We denote by $T_k$ a labelled tree homeomorphic to the cone of $k$ points as depicted in Figure~\ref{fig:corolla}.

\begin{figure}[ht]
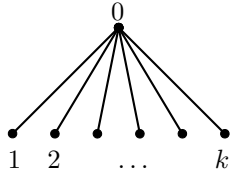
\caption{A labelled tree $T_k$ with only one vertex of valency $k\ge 3$}
\label{fig:corolla}
\end{figure}

Then it is known that $\mathbf{B}_n(T_k)$ is always a free group as follows.

\begin{lem}\label{lem:corolla}\cite{KKP}
The braid group $\mathbf{B}_n(T_k)$ is a free group of rank $r=r(n,k,k)$, where
\begin{equation}
\label{eq:rank}r(n,\nu,\mu)=(\nu-2)\binom{n+\mu-2}{n-1} - \binom{n+\mu-2}{n} -(\nu-\mu-1).
\end{equation}
\end{lem}

Especially, the 2-braid group $\mathbf{B}_2(T_k)$ is of rank $\binom{k-1}2$, indexed by $\{(i,j)|2\le i<j\le k\}$. Indeed, each pair $(i,j)$ corresponds to the loop $s_{i,j}$ in $B_2(T_k)$ as follows.

We first consider the tripod $T_3$ with the cone point $0$. Assume that two points $a$ and $b$ are initially lying on the edge $(1,0)$ and moreover $b$ is closer to $0$ than $a$.
Then we move $b$ to the second leaf and $a$ to the third leaf, and back $b$ to the initial position of $a$ and back $a$ to that of $b$. 
See Figure~\ref{fig:generator} and we will rigorously define this loop later in detail.
Then the loop $s$ defined in this way generates the infinite cyclic group which is actually $\mathbf{B}_2(T_3)$.
\begin{figure}[ht]
\[
s=\left(\vcenter{\hbox{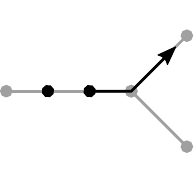}}\right)\cdot
\left(\vcenter{\hbox{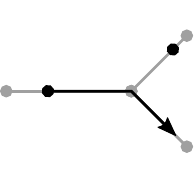}}\right)\cdot
\left(\vcenter{\hbox{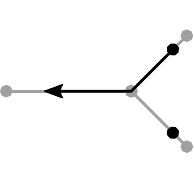}}\right)\cdot
\left(\vcenter{\hbox{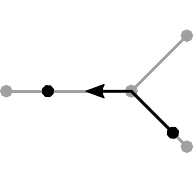}}\right)
\]
\caption{A loop $s$ in $B_2(T_3)$}
\label{fig:generator}
\end{figure}

For each pair $(i,j)$ with $i\neq j$, there exists a unique embedding $T_3\to T_k$ such that it maps $\partial T_3=\{1,2,3\}$ to $\{1,i,j\}\subset\partial T_k$ in order.
Then $s_{i,j}$ is nothing but the image of $s$ under the induced homomorphism $\mathbf{B}_2(T_3)\to\mathbf{B}_2(T_k)$. Note that $s_{j,i}$ is the inverse of $s_{i,j}$.

Now let $T$ be a tree with $k=\#(\partial T)$. We first label on $\partial T$ arbitrarily. Then there exists a unique label-preserving map $q:T\to T_k$ which takes a quotient by all internal edges and it induces a surjective homomorphism $q^*:\mathbf{B}_2(T_k)\to\mathbf{B}_2(T)$ by Proposition~\ref{prop:generaledgecontraction}. 
By the definition of $s_{i,j}$ above, the image $q^*(s_{i,j})$ coincides with the image of $s$ under the unique embedding $T_3\to T$ sending $\{1,2,3\}$ to $\{1,i,j\}$ as before.
We mean by the {\em center} $c(i,j)$ of $i$ and $j$ in $T$ that the image of $0\in T_3$ under this embedding.

\begin{figure}[ht]
\[
T=\vcenter{\hbox{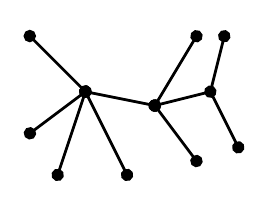}}\qquad\qquad 
\vcenter{\hbox{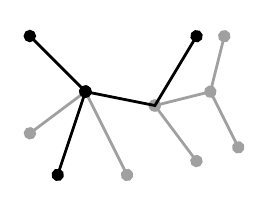}}
\]
\caption{A tree with labelled leaves and an embedding of $T_3$ corresponding to $s_{2,4}$}
\label{fig:orderedtree}
\end{figure}

Suppose that there is an isotopy $H:T_3\times I\to T$ 
such that $H_t(1)=1$ for all $t$ and $H_0(\{2,3\})=\{i,j\},H_1(\{2,3\})=\{i',j'\}$.
Then it defines a homotopy between the images of $s_{i,j}$ and $s_{i',j'}$ in $B_2(T)$, hence they are considered as the same in $\mathbf{B}_2(T)$.
More precisely, the given tree $T$ defines an equivalence relation on $\{(i,j)|2\le i<j\le k\}$ as follows.
\begin{defn}[Equivalence relation coming from a tree $T$ with ordered leaves]
Suppose the set $\partial T$ of leaves are indexed by $\{1,\dots,k\}$ and let $\binom{\partial T-1}2$ denote the set $\{(i,j)|2\le i<j\le k\}$.

Then we define an equivalence relation $\sim_T$ on $\binom{\partial T-1}2$ as $(i,j)\sim_{T}(i',j')$ if and only if
\begin{enumerate}
\item $c(i,j)=c(i',j')\in T$;
\item $[i]=[i']$ and $[j]=[j']$ in $\pi_0(T\setminus\{c(i,j)\})$.
\end{enumerate}
\end{defn}

Therefore, the equivalence classes depend not on the whole tree $T$ but only on the local shape, namely the {\em tangent space}, of each vetex of valency $\ge 3$. Hence each generator $s_{i,j}$ corresponds to a triple $(v, e_1, e_2)$ of a vertex $v$ with $\val(v)\ge 3$, and two half-edges $e_1$ and $e_2$ emitting from $v$ which are not heading to the chosen point in the boundary, $1\in\partial T$ in our example.

One can prove that this equivalence gives the complete set of defining relators for $\mathbf{B}_2(T)$, and therefore $\mathbf{B}_2(T)$ is free as well. 
\[\mathbf{B}_2(T)=\langle s_{2,3},\dots,s_{k-1,k}| s_{i,j}=s_{i',j'}\text{ if }(i,j)\sim_T(i',j')\rangle.\]

The rank is given by
\begin{equation}\label{eq:treerank}
r_2(T)=\sum_{v\in V(T)}\binom{\val(v)-1}2,
\end{equation}
where $V(T)$ is the set of vertices of $T$.
\end{ex}

\begin{rmk}
Recall the point appending map $\mathbf{B}_{n-1}(T)\to\mathbf{B}_n(T)$ defined in Proposition~\ref{prop:injection}, and consider all possible compositions which yield $\mathbf{B}_2(T)\to\mathbf{B}_n(T)$.
Then the images of $s_{i,j}$'s under these compositions generate $\mathbf{B}_n(T)$. 

More precisely, each generator is characterized by a vertex of valency $\ge 3$, two edges as before, and in addition the number of points in each component of the complement of that vertex in $T$. See \cite{KKP} for detail.
\end{rmk}

\subsection{Attaching higher cells}
We first consider the {\em generalized capping-off} $X$, which is to attach a $k$-simplex along a $(k-1)$-sphere of a given complex $Y$.
We exclude the cases when $(X,Y)=(D^2,S^1)$ or $(D^3,S^2)$ because 
their braid groups are already known, and moreover they are extremal in the sense of that the braid groups change dramatically before and after attaching simplices.

\begin{prop}\label{prop:highercell}
Let $X=Y\sqcup_\phi D^k$ via the embedding $\phi:\partial D^k=S^{k-1}\to Y$ for some $k\ge 3$ and a complex $Y$ different from $S^2$.
Then the embedding $Y\to X$ is a braid equivalence.
\end{prop}
\begin{proof}
We identify $\partial D^k$ with the subspace of $Y$ via $\phi$ from now on.

Let $*\in\mathring{D}^k$ be a point, and consider the subspace $B_{n-1;1}(X;*)$ of $B_n(X)$ consisting of configurations containing $*$.
Then this is of codimension $k\ge3$, in the sense that $B_{n-1;1}(X;*)\times\mathbb{R}^3$ can be embedded into $B_n(X)$.
Hence we may assume that all paths and homotopies in $B_n(X)$ are in general position with respect to $\{*\}$ and therefore they avoid $*$. In other words, the inclusion $X\setminus\{*\}\to X$ is a braid equivalence.

Let $D^k\setminus\{*\}\to\partial D^k$ be the radial projection, or the strong deformation retract, which naturally extends to $r:X\setminus\{*\}\to Y$, the homotopy inverse of the inclusion $Y\to X\setminus\{*\}$.

\begin{figure}[ht]
\[
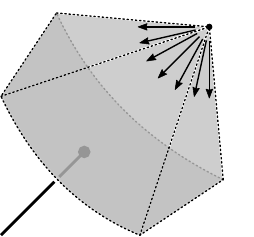\qquad\qquad
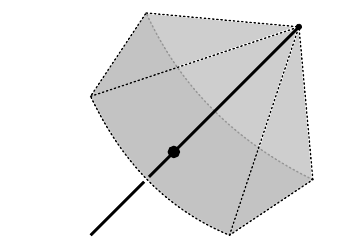
\]
\caption{A radial projection on $X\setminus\{*\}$ and an extended ray $[p,p+\epsilon]$}
\label{fig:radialprojection}
\end{figure}

Consider $B_n^{r\text{-fail}}=\{\mathbf{x}\in B_n(X\setminus\{*\})| \#(r(\mathbf{x}))<n\}$ consisting of configurations $\mathbf{x}$ such that at least two points in $\mathbf{x}$ are lying in a ray emitting from $*$ in $D^k$. Roughly speaking, it is the set of failures for $r$ to be extended to $\bar r:B_n(X\setminus\{*\})\to B_n(Y)$.
Then $B_n^{r\text{-fail}}$ is of codimension at least 2 in $B_n(X\setminus\{*\})$ as follows.
\begin{align*}
\codim\left(B_n^{r\text{-fail}}\subset B_n(X\setminus\{*\})\right) \ge \codim(\text{ray}\subset D^k\setminus\{*\})=(k-1)\ge 2.
\end{align*}

Hence by assuming the general position with respect to $B_n^{r\text{-fail}}$, we may assume that any loop misses $B_n^{r\text{-fail}}$ for all $k\ge 3$, and so does any disk for $k\ge 4$. Note that when $k=3$, a disk in $B_n(X\setminus\{*\})$ may intersect finitely many times with $B_n^{r\text{-fail}}$. Therefore, the map $r$ induces the surjective homomorphism 
\[
r^*:\mathbf{B}_n(Y)\to\mathbf{B}_n(X\setminus\{*\})\simeq\mathbf{B}_n(X),
\]
which is also injective if $k\ge 4$.

We claim that $r^*$ is an isomorphism for $k=3$ as well.

Suppose that $\partial D^3\subset\mathring{X}$, or equivalently, $\partial D^3$ is a component $\partial_0 Y$ of $\partial Y$. Then $X$ is an ordinary {\em capping-off} of the $2$-sphere $\partial_0 Y$ in $Y$, and so $Y\setminus \partial_0 Y\simeq X\setminus\{*\}$.
Since the homotopy equivalence between $Y$ and $Y\setminus\partial_0 Y$ induces the braid equivalence,  the inclusions $Y\to X\setminus\{*\}\to X$ induce
\[
\mathbf{B}_n(Y)\simeq\mathbf{B}_n(X\setminus\{*\})\simeq\mathbf{B}_n(X).
\]
Indeed, the strong deformation retract pushing $X\setminus\{*\}$ into $Y\setminus\partial_0 Y$, slightly smaller than $r$, induces the isomorphism $r^*$.

% or homeomorphic to the 3-ball $D^3$. The former case implies that 
Suppose that $\partial D^3\not\subset\mathring{X}$. Then since $Y\neq S^2$ by the hypothesis, $\partial D^3\not\subset\partial X$, and therefore $\partial D^3$ must intersect $\br(X)$.
%
%Since the latter case is excluded by the hypothesis, we assume that there exists $p$ in $\partial D^k\cap \br(Y)$.
The existence of a branch point $p\in\partial D^3\cap\br(X)$ implies that we can extend the ray $[*,p]$ emitting from $*$ passing through $p$ a little bit more. We denote the extended ray by $[p,p+\epsilon]$, where $p+\epsilon$ is a {\em point} lying in $\st(p)\setminus D^3$.
\begin{figure}[ht]
\[
\xymatrix@C=4pc{
\vcenter{\hbox{\def\svgscale{0.8}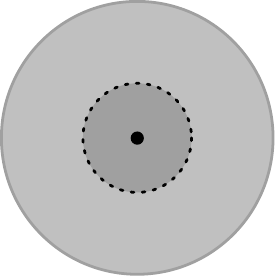}}\ar[r]^-{f=\{f_i\}}&
\vcenter{\hbox{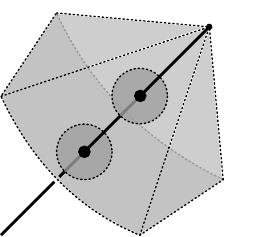}}\subset B_n(X\setminus\{*\})
}
\]
\caption{A homotopy disk and a small neighborhood $U$}
\label{fig:extendedray}
\end{figure}

Let $f=\{f_1(z),\dots,f_n(z)\}:(D^2,\partial D^2)\to(B_n(X\setminus\{*\}),B_n(Y))$ be given. To prove the claim, it suffices to show that $f$ can be homotoped into $B_n(Y)$.
Since $f$ is in general position with respect to $B_n^{r\text{-fail}}$, without loss of generality, we may assume that $f(D^2)$ intersects $B_n^{r\text{-fail}}$ exactly once at $0\in D^2$, and furthermore that there exists only one ray $[*,p']$ emitting from $*$, which contains exactly two points, say $f_1(0)$ and $f_2(0)$, among $f(0)$. Here $p'=r(f_1(0))=r(f_2(0))$.

Then we can further homotope $f$ by keeping $f$ in the general position so that $p'$ becomes $p$, and one of $f_1$ and $f_2$, say $f_1$, is constantly $p$ on a neighborhood $U\subset D^2$ of $0$.
The last comes from that in a small enough neighborhood, each $f_i$ can be homotoped separately.

Finally, we pull down $f_1$ on $U$ by using $[p,p+\epsilon]$ so that $f_1(0)\subset (p,p+\epsilon)$ as depicted in Figure~\ref{fig:sink}, and then $r\circ f:D^2\to B_n(Y)$ is well-defined and homotopic to $f$ relative to $\partial D^2$, as desired.
\end{proof}

We call the subcomplex $f^{-1}(B_n^{r\text{-fail}})$ of $D^2$ a {\em failure locus}. % and will see later.

\begin{figure}[ht]
\[
\xymatrix{
U=\vcenter{\hbox{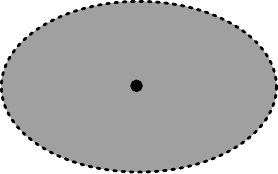}}\simeq
\vcenter{\hbox{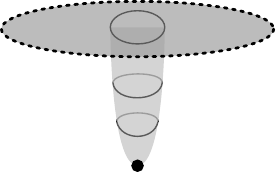}}\quad\ar[r]^-{f_1}&\quad
\vcenter{\hbox{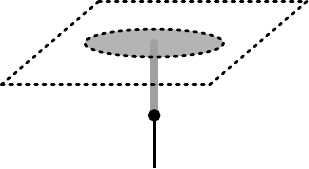}}\subset Y
}
\]
\caption{Pulling down a homotopy disk along the extended ray}
\label{fig:sink}
\end{figure}

\begin{rmk}\label{rmk:link1}
The effect of capping-off as above on a link $\lk(v)$ for $v\in S^{k-1}\subset Y$ is again a capping-off of $(k-2)$-sphere in $\lk(v)$ since $\lk(v)\cap S^{k-1}=S^{k-2}$ and $lk(v)\cap D^k=D^{k-1}$ in $X$.

Conversely, for any $v\in X$ and embedded sphere $S$ in $\lk(v)$, there exists a capping-off on $X$ which caps $\lk(v)$ off along $S$.
\end{rmk}

The direct consequence of the above proposition is as follows.
\begin{cor}\label{cor:2skeleton}
Let $X$ be a complex.
Then the embedding $X^{(2)}\to X$ of $2$-skeleton $X^{(2)}$ is a braid equivalence unless $X=D^3$ and $X^{(2)}=S^2$.
\end{cor}

Moreover, we can obtain the same result for the 2-cell attaching under the certain condition.
\begin{cor}\label{cor:2cell}
Let $X=Y\sqcup_\phi D^2$ via the embedding $\phi:\partial D^2\to Y$. 
Suppose $\phi(\partial D^2)$ bounds a disc $D'$ in $Y$, and $\mathring{D}'\cap \br(Y)\neq\emptyset$. 
Then the embedding $Y\to X$ is a braid equivalence.
\end{cor}
\begin{proof}
Note that there exists an embedded sphere $S=D\cup D'$ in $X$, such that 
\[
\mathring{D}'\cap \br(Y)\subset S\cap \br(X)\neq \emptyset.
\]
Hence $X$ satisfies the assumption of Proposition~\ref{prop:highercell}, and so $X\to X\sqcup_{S}D^3$ is a braid equivalence.

Then the embedding $Y\to X$ induces a surjection $\mathbf{B}_n(Y)\to\mathbf{B}_n(X)\simeq\mathbf{B}_n(X\sqcup_S D^3)$ as before, and moreover, in this case, we can think a strong deformation retraction $r'$ from $X\sqcup_S D^3$ to $Y$, which is nothing but an {\em elementary collapsing}. 

Then essentially the same argument as before with this elementary collapsing implies the braid equivalence of $Y\to X$. We omit the detail.
\end{proof}
In the last part of the proof, we are using the existence of a branch point in $\mathring{D}'$ again.

\begin{rmk}\label{rmk:link2}
Similar to the previous remark, the capping-off along $S^1\subset Y$ affects as the capping-off on $\lk(v)$ for any $v\in S^1$ along the $0$-sphere $S^0=\lk(v)\cap S^1$, and {\em vice versa}.
\end{rmk}

On the other hand, we can consider another type of embedding as follows.
Let $e=(v,w)$ be an edge of $X$ such that the closure $\overline{\st(v)}$ is homeomorphic to the boundary wedge sum $D^k\vee_\partial \bar e$, 
or equivalently, $\lk(v)=D^{k-1}\sqcup \{w\}$. 

Let $Y$ be a space obtained from $X$ by replacing $\overline{\st(v)}$ with $C_w(D^{k-1})$, where $C_w(D^{k-1})$ is a cone of $D^{k-1}\subset \lk(v)$ with the cone point corresponding to $w$. See Figure~\ref{fig:edgecontraction}. Then there is an obvious embedding $f:X\to Y$.
Note that $Y$ is homeomorphic to the quotient $X/\bar e$ but $f$ is different from the quotient map.

\begin{figure}[ht]
\[
\xymatrix{
\qquad\qquad X\ar[r]^f&X/\bar e\qquad\qquad \\
\overline{\st(v)}=\vcenter{\hbox{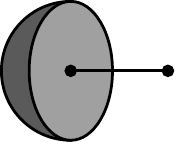}}\ar[r]\ar@<-1.7pc>[u]&
\vcenter{\hbox{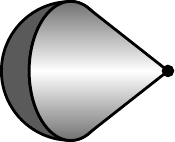}}=C_w(D^{k-1})\ar@<1.7pc>[u]
}
\]
\caption{An embedding having the same effect as the edge contraction}
\label{fig:edgecontraction}
\end{figure}

\begin{prop}\label{prop:edgecontraction}
Let $X$ be given and $e=(v,w)$ be an edge in $X$ with $\lk(v)=D^{k-1}\sqcup \{w\}$ for some $k\ge 2$.
Then the embedding $f:X\to X/\bar e$ defined as above is a braid equivalence.
\end{prop}
\begin{proof}
We first endow a metric $d$ on $X$, and the induced metric $d'$ on $X/\bar e$. Assume that $d(v,w)=diam(\bar e)=1$.
Let $f_\epsilon:X\to X/\bar e$ for $0\le\epsilon<1$ be an embedding such that $diam'(f_\epsilon(\bar e))=1-\epsilon$ and
for any $0<\epsilon<\epsilon'<1$,
\[
f(X)=f_0(X)\subsetneq f_\epsilon(X)\subsetneq f_{\epsilon'}(X)\subsetneq X/\bar e
\]
as depicted in Figure~\ref{fig:deformation}.

For convenience sake, we define $f_1$ by the quotient map $X\to X/\bar e$,
and so $X/\bar e=\varinjlim_{\epsilon\in [0,1)} f_\epsilon(X)$. This also implies that
\begin{equation}
\label{eqn:limit}
B_n(X/\bar e) = \varinjlim_{\epsilon\in [0,1)} f_\epsilon(B_n(X)).
\end{equation}

Moreover, since all $f_\epsilon(X)$'s are ambient isotopic in $X/\bar e$, so are $f_\epsilon(B_n(X))$'s in $B_n(X/\bar e)$. Especially, the inclusion $f_\epsilon(B_n(X))\subset f_{\epsilon'}(B_n(X))$ is a homotopy equivalence for any $\epsilon<\epsilon'<1$.
Hence by this fact and (\ref{eqn:limit}), it suffices to show that any given $c:(D^m,\partial D^m)\to(B_n(X/\bar e), f_0(B_n(X)))$ factors through the inclusion 
$f_\epsilon(B_n(X))\to B_n(X/\bar e)$ for some $\epsilon<1$ up to homotopy.

Since the image of $c$ is compact, we can choose a constant $\epsilon$ such that
\[0<1-\epsilon<\frac 13 \min_{x\in D^m} \min \{ d'(x_i,x_j) | x_i\neq x_j\in c(x)\}.\]
Now let $r_\epsilon:X/\bar e\times[0,1]\to X/\bar e$ be a strong deformation retraction of $X/\bar e$ onto $f_\epsilon(X)$. Then the composition $r_\epsilon\circ c$ is a well-defined homotopy between $c$ and a map into $f_\epsilon(B_n(X))$. This completes the proof.
\end{proof}

\begin{figure}[ht]
\[
f_0(X)=\vcenter{\hbox{\includegraphics{diskwithedge1.pdf}}}\subsetneq
\vcenter{\hbox{\includegraphics{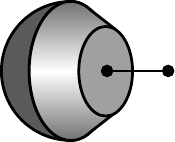}}}\subsetneq
\vcenter{\hbox{\includegraphics{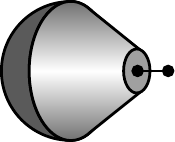}}}\subsetneq
\vcenter{\hbox{\includegraphics{diskwithedge2.pdf}}}=f_1(X)
\]
\caption{Local pictures of $f_0(X), f_{\epsilon}(X),f_{\epsilon'}(X)$ and $f_1(X)$ for $0<\epsilon<\epsilon'<1$}
\label{fig:deformation}
\end{figure}

\subsection{Simple complex}
For given $X$, we want to find a {\em simple} complex $X'$ whose braid groups are isomorphic to those on $X$. To do this, we need the following proposition which is the 2-dimensional analogue of Proposition~\ref{prop:edgecontraction}.
\begin{prop}\label{prop:graphedgecontraction}
Let $X$ be a complex of dimension $2$ and $e=(v,w)$ be an edge in $X$ with $\lk(v)=\Gamma\sqcup\{w\}$ for some connected graph $\Gamma$. 
If $\Gamma=S^1$, we assume furthermore that $w\not\in\partial X$.
Then there exist braid equivalences 
\[X\to \overline X\to\overline{X}/\bar{e}\leftarrow X/\bar e\]
for some complex $\overline X$, and therefore $X\equiv_B X/\bar e$.
\end{prop}
\begin{proof}
Suppose $\Gamma$ is trivial. Then we set $\overline{X}$ to be $X$ itself.
If $\Gamma=I$, then this is a special case of Proposition~\ref{prop:edgecontraction}.
%If $\Gamma=S^1$, then we first cap off $\Gamma$ and fill a 3-cell in the sphere produced by the capping-off. Then $\lk(v)=D^2\sqcup\{w\}$ and apply Proposition~\ref{prop:edgecontraction}.

% and $w\not\in\partial X$, then $X$ is homeomorphic to $X/\bar{e}$ and there is nothing to prove.

We assume that $\Gamma\neq I$.
Then the strategy is as follows. 

We first attach cells to $X$ near $v$ without touching $e$ to obtain $X'$ so that $X\to \overline{X}$ is a braid equivalence and $\lk(v)=D^k\sqcup\{w\}$ in $\overline{X}$ for some $k$. This can be done by Proposition~\ref{prop:highercell} and Corollary~\ref{cor:2cell} since $v$ is a branch point and it always satisfies the assumption of Corollary~\ref{cor:2cell}.
Then there exists a braid equivalence $\overline{X}\to\overline{X}/\bar{e}$ by Proposition~\ref{prop:edgecontraction}.

Notice that $\bar{X}/\bar{e}$ can be obtained from $X/\bar{e}$ by attaching cells in the exactly same ways as we did for $X$ to obtain $\overline{X}$.
Hence we have a braid equivalence $X/\bar{e}\to \overline{X}/\bar{e}$, and therefore there exist braid equivalences $X\to \overline{X}\to \overline{X}/\bar{e}\leftarrow X/\bar{e}$ as claimed.

Recall the effects of attaching cells on $\lk(v)$ as mentioned in Remark~\ref{rmk:link1} and \ref{rmk:link2}, which are capping-off along embedded spheres in $\lk(v)$.
Hence, it suffices to show that $\lk(v)$ can be transformed to $D^k$ by the iterated capping-off process, and this is actually equivalent to showing that $\lk(v)$ is a subset of the 1-skeleton $K^{(1)}$ for some simplicial complex $K$ homeomorphic to $D^k$.

Since any graph embeds into $\mathbb{R}^3$, it is always possible for $k=3$ and so is it for $k=2$ when $\Gamma$ is planar. This completes the proof.
\end{proof}

\begin{proof}[Proof of Theorem~\ref{thm:simple}]
Let $X_1$ be the 2-skeleton of a sufficiently subdivided complex $X$. Then by Corollary~\ref{cor:2skeleton}, $X\equiv_B X_1$.

Let $w\in \br(X_1)$ be a non-simple vertex.
That is, $\lk(w)$ has at least 2 graph components or only one graph component with $\val(w)\ge 3$.
Let $\Gamma$ be a graph component of $\lk(w)$ and $X_2$ be a complex having an edge $e=(v,w)$ such that $X_1=X_2/\bar e$ and $\lk(v)=\Gamma\sqcup\{w\}$.
Then $X_1\equiv_B X_2$ by Proposition~\ref{prop:graphedgecontraction}.

Since $v$ is simple, and $\val(w)$ in $X_2$ is equal to $\val(w)$ in $X_1$ but the number of graph components of $\lk(w)$ in $X_2$ is strictly less than that in $X_1$. 
Therefore by the induction on the number of graph components in nonsimple vertices, we eventualy obtain a simple complex $X'$ which is braid equivalent to $X$.
\end{proof}

\section{Decompositions and Elementary complexes}\label{sec:elementary}

Let $X$ and $Y$ be simple complexes. We define two operations on $X$ and a pair $X, Y$ as follows.

\begin{defn}[$k$-closure]
Let $X$ be a {\em connected} complex and $\mathbf{v}=\{v_1,\dots,v_k\}\subset\partial X$. A {\em $k$-closure $\widehat{(X,\mathbf{v})}$ of $X$ along $\mathbf{v}$} is a complex obtained by the mapping cone of the embedding $\mathbf{v}\to X$, and called {\em trivial} if %$k=1$ and $\overline{\st(v)}$ is trivial.
$\overline{\st(v)}$ is trivial, or equivalently, $k=1$ and $\overline{\st(v)}=I$.
\end{defn}

%Here, we never consider closures of disconnected complexes. 
Notice that $\widehat{(X,\mathbf{v})}$ can be also obtained by gluing $T_k$ to $X$ along $\mathbf{v}$, 
and $\widehat{(X,\mathbf{v})}=X$ if and only if it is a trivial closure. Moreover, if $\overline{\st(v)}=T_k$ for some $k\ge 1$ and $X_v$ is connected, then the closure 
of $X_v$ along $\lk(v)$ becomes $X$ itself by definition of $X_v$. 
Hence $X_v$ is a kind of the inverse of the closure operation, usually called {\em unwrapping} in graph theory. 
Especially, we denote by $\Theta_k$ {\em the closure $\widehat{(T_k,\partial T_k)}$ of $T_k$ along $\partial T_k$}, which is the union of two $T_k$'s and has two distinguished vertices $0$ and $0'$ of valency $k$.

We will suppress $\mathbf{v}$ unless any ambiguity occurs.

Let $v\in X$ be a vertex with $\overline{\st(v)}=T_k$ for some $k\ge 1$. Then we denote by $\vec v$ a vertex $v$ together with an ordering on $\lk(v)$. In order words, we may identify $\st(v)$ with the labelled $T_k$, and we regard $\lk(v)=\{v_1,\dots,v_k\}$ as an ordered $k$-tuple $(v_1,\dots,v_k)$ if $\vec v$ is given.
We call $\vec v$ {\em a vertex with ordering}.

\begin{defn}[$k$-connected sum]
Let $X$ and $Y$ be complexes, and $\vec v$ and $\vec w$ be vertices of orderings of valency $k\ge 1$ in $X$ and $Y$.
We further assume that both $X_v$ and $Y_w$ are connected.

A {\em $k$-connected sum $X\# Y$ of $(X,\vec v)$ and $(Y,\vec w)$} is a complex obtained by connecting each $v_i$ and $w_i$ via an interval $e_i$, 
and is called {\em trivial} if one of $X$ and $Y$ is $\Theta_k$.
\end{defn}

See Figure~\ref{fig:connectedsum} for a pictorial definition for connected sum.
Note that $(X,\vec v)\# (Y,\vec w)$ is a boundary wedge sum $X\vee_\partial Y$ if $k=1$, and is an ordinary connected sum if $k=2$.

\begin{figure}[ht]
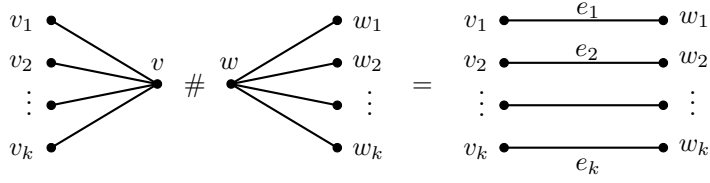
\caption{A $k$-connected sum}
\label{fig:connectedsum}
\end{figure}

Since $\Theta_k$ is a $k$-closure of $T_k$, for any vertex $v\in X$ with $\overline{\st(v)}=T_k$ for some $k\ge 1$, a $k$-connected sum $(X,\vec v)\#(\Theta_k,\vec 0)$ is nothing but a $k$-closure of $X_v$ along $\lk(v)$ and so it is $X$ itself whatever the orderings on $v$ and $0$ are. Therefore $\Theta_k$ plays the role of the identity under the $k$-connected sum.

We sometimes suppress $\vec v$ and $\vec w$ unless any ambiguity occurs, and also say that $X$ is {\em decomposed into $Y$ and $Z$ via $k$-connected sum} if $X=Y\# Z$.

Furthermore, it is easy to see that both unwrapping and connected sum decomposition reduce the first Betti number or the number of vertices of connected components. 
Therefore by continuing these operation, we eventually have components which are elementary in the sense of the following definition.

\begin{defn}[Elementary complex]
Let $X$ be a simple complex of dimension 2. We say that $X$ is {\em elementary} if $X$ can be expressed as neither a nontrivial $k$-closure nor a nontrivial $k$-connected sum.
\end{defn}

Let $X$ be an elementary complex. 
Suppose $\dim X=1$. Then elementariness forces $X$ to be a tree having at most 1 vertex of valency $k\ge 3$. Therefore $X$ is homeomorphic to $T_k$, which admits a trivial closure structure only, and so is elementary.

Suppose $\dim X=2$ and $f:X\to M$ is a simplicial embedding into a piecewise linear manifold $M$. Assume that $\dim M$ is minimal among all possible such embeddings.
Then $\dim M\le 4$ since $\dim X=2$. 

Let $N(X)$ be the closed regular neighborhood of $X$ in $M$, or equivalently, $\overline{\st(X)}$ in $M$ after sufficient barycentric subdivisions.
Hence $N(X)$ can be obtained by attaching 2 or higher dimensional cells to $X$. Roughly speaking, $N(X)$ is a {\em thickening} of $X$.
Note that $N(X)$ depends on $f$ but $\dim N(X)$ does not.

If $\dim N(X)=2$, or equivalently, $X$ can be embedded into a surface, then we claim that there is no branch point in $X$ and so $X$ itself is a surface $\Sigma=N(X)$. 

Suppose $v\in \br(X)$. If $\lk(v)$ is 0-dimensional, then $\overline{\st(v)}\simeq T_k$ for $k=\val(v)\ge 3$ and therefore $X$ can be decomposed further via a nontrivial $k$-closure or a nontrivial $\ell$-connected sum for some $\ell< k$.
This contradicts to elementariness of $X$ and so $\overline{\st(v)}$ is homeomorphic to a boundary wedge sum $D^2\vee_\partial [v,w)$ of a disk and an half-open edge $[v,w)$ since $X$ is simple and embeds into a surface. Hence either $w$ is a point of 2-closure of $X_w$ when $X_w$ is connected, or $v$ is a point of 1-closure or 1-connected sum of $X$ otherwise. This does not happen by the elementariness of $X$, and so $\overline{\st(v)}=D^2$. Therefore $v\not\in \br(X)$, and this contradiction implies that $\br(X)=\emptyset$.

For a surface $\Sigma$, we omit the group presentation of $\mathbf{B}_n(\Sigma)$ which is well-known, and actually we already introduced the result we need in Theorem~\ref{thm:surface}.

On the other hand, if $\dim N(X)\ge3$, then by the same argument as above, all vertices of $\br(X)$ are of valency 1. In this case, we call $X$ a {\em branched surface}.
Moreover, we can attach cells of dimension at least 2 to $X$ to obtain $N(X)$ so that the inclusion $X\to N(X)$ becomes a braid equivalence by Proposition~\ref{prop:highercell} and Corollary~\ref{cor:2cell}. 
This is essentially the same process as described in the proof of Theorem~\ref{thm:simple}.
Therefore
\[\mathbf{B}_n(X)\simeq\mathbf{B}_n(N(X))\simeq\prod^n\pi_1(N(X))\rtimes\mathbf{S}_n\simeq\prod^n\pi_1(X)\rtimes\mathbf{S}_n.\]

\begin{lem}\label{lem:elementary}
Let $X$ be an elementary complex. Then $X$ is either $T_k$, a surface $\Sigma$, or a branched surface.
Moreover, if $X$ is a branched surface, then 
\[\mathbf{B}_n(X)\simeq\prod^n\pi_1(X)\rtimes\mathbf{S}_n.\]
\end{lem}

\begin{ex}[An elementary, non-manifold complex $S_0$]\label{ex:S_0}
Let $S_0$ be the complex obtained by gluing a disk and an interval, depicted in Figure~\ref{fig:S_0}.
\[S_0=\{(x,y,0)|-1\le x, y\le 1\}\cup \{(0,0,z)|0\le z\le 1\}\subset \mathbb{R}^3.\]
Then it is obvious that $S_0$ is elementary and $\dim N(S_0)=3$ since $S_0$ can not be embedded into any surface. 
Hence $N(S_0)\simeq D^3$ and therefore $\mathbf{B}_n(S_0)$ is isomorphic to $\mathbf{S}_n$ via the induced permutation $\rho$.

It is not hard to see that an elementary complex $X$ embeds into a surface if and only if it does not contain $S_0$.
Furthermore, the same holds for non elementary complexes. This is easy and we will see later.
In this sense, $S_0$ is the obstruction complex for given complex to be embedded into a surface.
\end{ex}

In the following two sections, we will present the braid groups on $\widehat X$ and $X\#Y$ in terms of the braid groups on $X$ and both $X$ and $Y$, respectively.

Let $X$ and $Y$ be connected, disjoint subspaces of $Z$. Then for convenience's sake, we denote by $B_{r;s}(X;Y)$ the subspace of $B_{r+s}(Z)$ defined as
\[B_{r;s}(X;Y)=\left\{\left.\mathbf{x}\in B_{r+s}(Z)\right| \#(\mathbf{x}\cap X)=r, \#(\mathbf{x}\cap Y)=s\right\}\simeq B_r(X)\times B_s(Y).\]
Hence $\pi_1(B_{r;s}(X;Y))=\mathbf{B}_r(X)\times\mathbf{B}_s(Y)$ and is denoted by $\mathbf{B}_{r;s}(X;Y)$.

\section{\texorpdfstring{$k$}{k}-Closure of a complex}
Let $X$ be a complex and $\mathbf{v}=\{v_1,\dots,v_k\}\subset \partial X$. Let $v$ be the cone point of the $k$-closure $\widehat X$ of $X$ along $\mathbf{v}$, which is the mapping cone of $\mathbf{v}\to X$.
We denote by $e_i$ the {\em oriented} edge $(v_i,v)$ from $v_i$ to $v$ in $\widehat X$. Then for each $1\le i\le k$, the concatenation $e_1e_i^{-1}$ defines a path $\delta_i$ in $B_n(\widehat X)$ from $i_{v_1}(\mathbf{x})$ to $i_{v_i}(\mathbf{x})$ for any $\mathbf{x}\in B_{n-1}(X)$ in the obvious way, and we denote this path by $\delta_i$ again unless any ambiguity occurs. Obviously, $\delta_1$ defines a path homotopic to a constant path.

Now we endow a metric $d$ on $\widehat X$. Then there is a constant $\epsilon=\epsilon(\widehat X)$ as discussed earlier so that $d(x_i,x_j)\ge \epsilon$ for any $x_i, x_j\in\mathbf{x}\in B_n(\widehat X)$.
Then by subdividing all edges adjcent to $v$, we may assume that the diameter of $\overline{\st(v)}$ is less than $\epsilon$.
In other words, we may assume that any configuration $\mathbf{x}=B_n(\widehat X)$ intersects $\overline{\st(v)}$ at most once, that is, $\#\left(\mathbf{x}\cap\overline{\st(v)}\right)\le 1$. 

Therefore, $B_n(\widehat X)$ is separated into two subsets according to the presence of a point in $\overline{\st(v)}$, and is the union of two subspaces $B_{n-1;1}\left(X\setminus\mathbf{v};\overline{\st(v)}\right)$ and $B_n(X)$ whose intersection is 
\begin{equation}\label{eq:subspace}
B_{n-1;1}(X\setminus\mathbf{v};\mathbf{v})=\bigsqcup_{i=1}^k \left(B_{n-1}(X\setminus\mathbf{v})\times\{v_i\}\right).
\end{equation}
Notice that the intersection is not connected by the assumption that there is at most 1 point in $\overline{\st(v)}$.
Hence we need to choose paths joining components to make it connected, and make the Seifert-van Kampen theorem applicable.

To this end, we fix a basepoint $*_\ell$ of $B_\ell(X)$ for each $\ell\le n$ such that $i_{v_1}(*_{n-1})=*_n$.
Then we choose a path $\gamma_i$ in $B_n(X)$ for each $1\le i\le k$ between $*_n$ and $i_{v_i}(*_{n-1})$ such that $\gamma_i(t)$ avoids $\mathbf{v}$ for $0<t<1$ and $\bigcup_{i=1}^k \gamma_i\subset B_n(X)$ is homotopy equivalent to $T_k$.
In other words, if $\gamma_i$ and $\gamma_j$ intersect at some point, then the images of $\gamma_i$ and $\gamma_j$ must coincide from the beginning. Since $i_{v_1}(*_{n-1})=*_n$, we may assume that $\gamma_1$ is a constant path at $*_n$ for convenience sake.
%Remark that $\gamma_i$'s interfere each other and may affect the subspaces $i_{v_j}(B_{n-1}(X))$ in general. 

In practice, the most convenient way to choose $\gamma_i$'s is as follows.
At first, we fix a set $\{\gamma^0_2,\dots,\gamma^0_k\}$ of paths in $B_n(T_k)$ as depicted in Figure~\ref{fig:gamma}.

\begin{figure}[ht]
\[
*_n=\vcenter{\hbox{\def\svgscale{0.7}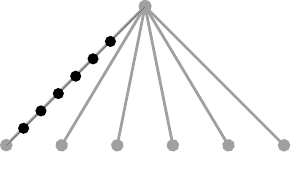}}
\xrightarrow{\gamma_i^0=\vcenter{\hbox{\def\svgscale{0.7}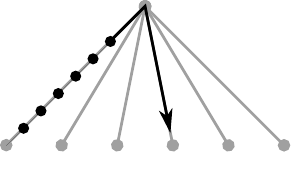}}\quad}
\vcenter{\hbox{\def\svgscale{0.7}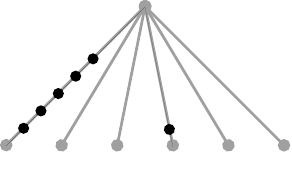}}=i_{v_i}(*_{n-1}),
\]
\caption{A path $\gamma_i^0$ for $T_k$ joining $*_n$ and $i_{v_i}(*_{n-1})$}
\label{fig:gamma}
\end{figure}

Since $X$ is connected, there exists a tree $T\subset X$ with $\partial T = \mathbf{v}$. 
Then there is a map $q:T\to T_k$ which contracts all internal edges of $T$ and induces a homotopy equivalence.
Hence similar to Proposition~\ref{prop:generaledgecontraction}, we can find a lift $\gamma_i$ for each $\gamma_i^0$. In this case, points in $*_n$ are lying near $v_1$. 

\begin{lem}\label{lem:thetasubgroup}
Let $X, \mathbf{v}$ as above. Then there exists a homomorphism
\[\Psi_{\widehat X}:\mathbf{B}_n(\Theta_k)\to\mathbf{B}_n(\widehat X).\]
\end{lem}
\begin{proof}
Let $T$ be as above. Then by gluing $T_k$ to both $T$ and $T_k$, we have a map 
\[\hat q: \widehat{T}\to \Theta_k,\]
where $\widehat{T}$ is an obtained graph homotopy equivalent to $\Theta_k$.
By Proposition~\ref{prop:generaledgecontraction}, it induces a surjective map $\hat q^*:\mathbf{B}_n(\Theta_k)\to\mathbf{B}_n(\widehat{T})$.
Then the desired map $\Psi_{\widehat X}$ is just a composition of $\hat q^*$ and the map induced by the inclusion $\widehat{T}\to \widehat X$.
\end{proof}

It is important to remark that $\Psi_{\widehat X}$ is neither injective nor surjective in general by the same reason as stated in the discussion after Proposition~\ref{prop:generaledgecontraction}.

Let $\widehat{B}_{n-1;1}\left(X\setminus\mathbf{v};\overline{\st(v)}\right)=\left(\bigcup_i\gamma_i\right)\cup B_{n-1;1}\left(X\setminus\mathbf{v}; \overline{\st(v)}\right)$.
Then for each $i$, since 
\[
\gamma_i\cap B_{n-1;1}\left(X\setminus\mathbf{v};\overline{\st(v)}\right)=\{*_n, i_{v_i}(*_{n-1})\}
\]
and $*_n$ and $i_{v_i}(*_{n-1})$ is connected via the path $\delta_i^{-1}\delta_1$ in $\overline{\st(v)}$, and so $\gamma_i$ together with $\gamma_1$ defines a loop $\gamma_i\delta_i^{-1}\delta_1\gamma_1^{-1}$.
Hence $\bigcup_i\gamma_i$ contributes $(k-1)$ loops and so $(k-1)$ free letters in the fundamental group.

More precisely, let $t_i=[\gamma_i\delta_i^{-1}\delta_1\gamma_1^{-1}]$ denote the homotopy class, and we set $t_1=e$.
Then
\begin{align*}
\pi_1\left(\widehat{B}_{n-1;1}\left(X\setminus\mathbf{v};\overline{\st(v)}\right),*_n\right)&\simeq
\mathbf{B}_{n-1;1}\left(X\setminus\mathbf{v};\overline{\st(v)}\right)\ast\langle t_2,\dots,t_k\rangle\\
&\simeq\left(\mathbf{B}_{n-1}(X\setminus\mathbf{v})\times \pi_1\left(\overline{\st(v)}\right)\right)\ast\langle t_2,\dots,t_k\rangle\\
&\simeq\mathbf{B}_{n-1}(X)\ast\langle t_2,\dots,t_k\rangle.
\end{align*}

On the other hand, the intersection between $\widehat{B}_{n-1;1}\left(X\setminus\mathbf{v};\overline{\st(v)}\right)$ and $B_n(X)$ is precisely $\left(\bigcup_i\gamma_i\right)\cup B_{n-1;1}(X;\mathbf{v})$, and homotopy equivalent to a wedge sum of $k$-copies of $B_{n-1}(X)$ indexed by $v_i$'s as shown in (\ref{eq:subspace}).
Hence its fundamental group is isomorphic to $\ast^k_{i=1}\mathbf{B}_{n-1}(X)$.
Moreover, for each $i$, there are two inclusions $\widehat{\phi}_i$ and $\psi_i$ from $\mathbf{B}_{n-1}(X)$ to 
$\mathbf{B}_{n-1}(X)\ast\langle t_2,\dots,t_k\rangle$ and $\mathbf{B}_n(X)$ defined as
for all $\beta\in\mathbf{B}_{n-1}(X)$, as paths
\[\widehat{\phi}_i(\beta)=\psi_i(\beta)=\gamma_i\cdot i_{v_i}(\beta)\cdot\gamma_i^{-1}.\]
However, as group elements,
\[\widehat{\phi}_i(\beta)=t_i\beta t_i^{-1},\quad\psi_i(\beta)=v_{i*}(\beta),\]
where $v_{i*}=(\gamma_i)_*^{-1}(i_{v_i})_*$ and $(\gamma_i)_*$ is the automorphism changing the basepoint from $*_n$ to $i_{v_i}(*_{n-1})$.

Therefore we have a diagram below whose push-out defines $\mathbf{B}_n(\widehat X)$ by the Seifert-van Kampen theorem.
\begin{equation}\label{eq:pushout}
\xymatrix{
\mathbf{B}_{n-1}(X)\ast\langle t_2,\dots,t_k\rangle & &\mathbf{B}_n(X)\\
&\mathop{\ast}_{i=1}^k\mathbf{B}_{n-1}(X)\ar[lu]^{\ast_{i=1}^k\widehat{\phi}_i}\ar[ru]_{\ast_{i=1}^k\psi_i}
}
\end{equation}

Hence, 
\begin{align*}
\mathbf{B}_n(\widehat X)&=\frac{\left(\mathbf{B}_{n-1}(X)\ast\langle t_2,\dots,t_k\rangle\right)\ast\mathbf{B}_n(X)}
{\left\langle\!\!\left\langle
\widehat{\phi}_i(\beta)=\psi_i(\beta),\forall \beta\in\mathbf{B}_{n-1}(X),1\le i\le k
\right\rangle\!\!\right\rangle}\\
&=\frac{\mathbf{B}_n(X)\ast\langle t_2,\dots,t_k\rangle}
{\langle\!\langle
t_i\beta t_i^{-1}=v_{i*}(\beta),\forall \beta\in\mathbf{B}_{n-1}(X),2\le i\le k
\rangle\!\rangle}.
\end{align*}

The last equality follows by identifying $\mathbf{B}_{n-1}$ as a subgroup of $\mathbf{B}_n(X)$ via $\widehat{\phi}_1$ and $v_{1*}$.
Note that when $k=2$, then $\mathbf{B}_n(\widehat X)$ is an ordinary {\em HNN extension} of $\mathbf{B}_n(X)$ with the associated group $\mathbf{B}_{n-1}(X)$. 

\begin{thm}\label{thm:closure}
Let $X$ be a complex and $\mathbf{v}=\{v_1,\dots,v_k\}\subset \partial X$.
Then the braid group $\mathbf{B}_n(\widehat X)$ on the $k$-closure $\widehat X$ of $X$ along $\mathbf{v}$ is as follows. For $n\ge 1$,
\[\mathbf{B}_n(\widehat X)=\frac{\mathbf{B}_n(X)\ast\langle t_2,\dots,t_k\rangle}
{\langle\!\langle
t_i\beta t_i^{-1}=v_{i*}(\beta),\forall\beta\in\mathbf{B}_{n-1}(X),2\le i\le k
\rangle\!\rangle},\]
where 
\[v_{i*}=(\gamma_i)_*^{-1}(i_{v_i})_*:\mathbf{B}_{n-1}(X)\to\mathbf{B}_n(X),\]
and $\gamma_i$ is a chosen path joining basepoints of $B_n(X)$ and $i_{v_i}(B_{n-1}(X))$.
Moreover, $\mathbf{B}_{n-1}(X)$ is identified with a subgroup of $\mathbf{B}_n(X)$ via $v_{1*}$.
\end{thm}

Let $\partial_0 X$ be a connected component of $\partial X$ of dimension at least 1, and suppose $\{v_1,\dots, v_k\}\subset \partial_0 X$.
Then since $\st(\partial_0 X)$ is of dimension at least 2, we may choose $\gamma_i$'s in a small enough collar neighborhood of $\partial_0 X$ as depicted in Figure~\ref{fig:paths} so that they intersect pairwise only at $v_1$.
Then 
each path $\gamma_i$ can be regarded as disjoint from $B_{n-1}(X)$. 
This implies the triviality of the action $(\gamma_i)_*$ on $(i_{v_i})_*(\mathbf{B}_{n-1}(X))$. Hence $v_{i*}(\beta)=\beta$ for all $\beta\in\mathbf{B}_{n-1}(X)$, and the defining relator is nothing but the commutativity between $t_i$ and any $\beta\in\mathbf{B}_{n-1}(X)$.

\begin{cor}\label{cor:boundary}
Let $X, \mathbf{v}$ be as above. Suppose $\dim X\ge 2$ and all $v_i$'s are lying in the same component of $\partial X$. Then
\[\mathbf{B}_n(\widehat X)\simeq \mathbf{B}_n(X)\ast\langle t_2,\dots, t_k\rangle/
\langle\!\langle[\mathbf{B}_{n-1}(X), t_i],2\le i\le k\rangle\!\rangle.\]
\end{cor}

\begin{figure}[ht]
\[
\widehat X=\vcenter{\hbox{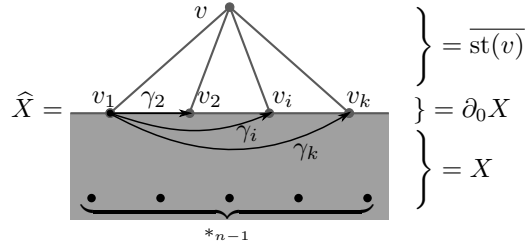}}
\]
\caption{A choice of paths $\{\gamma_i\}$ when $\{v_1,\dots,v_k\}\subset\partial_0 X$}
\label{fig:paths}
\end{figure}

\begin{rmk}\label{rmk:nonplanar}
On the other hand, for a surface $\Sigma$, if we take a closure $\widehat \Sigma$ along points not contained in a single boundary component of $\Sigma$, then $\widehat \Sigma$ is always nonplanar. 
Moreover it contains a nonplanar graph.
\end{rmk}

\begin{ex}[2 braid group on the closure of a tree]\label{ex:treeclosure}
Let $T$ be a tree with $k=\#(\partial T)$ and $\widehat T$ be the $k$-closure of $T$ along $\partial T$. Since $\mathbf{B}_1(T)$ is trivial, $\mathbf{B}_2(\widehat T)$ is a free group and admits the following presentation.
\begin{align*}
\mathbf{B}_2(\widehat T)&=\mathbf{B}_2(T)\ast \pi_1(\Theta_k)\\
&=\langle
s_{2,3},\dots,s_{k-1,k},t_2,\dots,t_k|s_{i,j} = s_{i',j'}\text{ if }(i,j)\sim_T(i',j')
\rangle.
\end{align*}
Hence the rank is $r_2(T)+(k-1)$, where $r_2(T)$ is the rank of $\mathbf{B}_2(T)$ given by the formula (\ref{eq:treerank}) in Example~\ref{ex:tree}.
\end{ex}

\begin{ex}[The braid group on $\Theta_k$]\label{ex:theta}
Since $\Theta_k=\widehat T_k$, this is a special case of the previous example.
Recall that $T_k$ produces only the trivial equivalence relation on $\binom{\partial T_k-1}2$. Therefore 
$\mathbf{B}_2(\Theta_k)$ is a free group of rank $\binom{k-1}2+(k-1)=\binom{k}2$.

Now we consider $\mathbf{B}_n(\Theta_k)$ which is generated by $t_i$'s and $\mathbf{B}_n(T_k)$. More specifically, we have a diagram
\[
\xymatrix@C=3pc{
\mathbf{B}_2(T_k)\ar[r]^{(i_{v_1})_*^{n-2}} \ar[d]_{\widehat{(\cdot)}_*} & \mathbf{B}_n(T_k) \ar[d]^{\widehat{(\cdot)}_*}\\
\mathbf{B}_2(\Theta_k)\ar[r]^{\xi}&\mathbf{B}_n(\Theta_k),
}
\]
where the vertical arrows are induced by the inclusion $\widehat{(\cdot)}:T_k\to\Theta_k$, and $\xi$ is defined by
\[
\xi(\sigma_{i,j}) = \left(\widehat{(\cdot)}_*\circ (i_{v_1})_*^{n-2} \right) (\sigma_{i,j}),\quad
\xi(t_i) = t_i\in\mathbf{B}_n(\Theta_k).
\]

Note that $\xi$ is well-defined since $\mathbf{B}_2(\Theta_k)$ is free, and commutativity is obvious by the definition of $\xi$.
\begin{lem}\cite{KP}
The map $\xi$ is surjective.
\end{lem}
It is not hard to prove this lemma. Indeed, one can prove this by drawing carefully the paths representing the generators for $\mathbf{B}_n(T_k)$ and $t_i$'s.
In general, for any tree $T$ with $k=\#(\partial T)$, there is a surjective homomorphism $\mathbf{B}_2(\widehat T)\to\mathbf{B}_n(\widehat T)$.
\end{ex}

\begin{rmk}
The decomposition defined above is nothing but a {\em graph-of-groups} structure for $B_n(X)$ over the graph $\Theta_k$ as follows. Note that this is essentially same as the push-out diagram in (\ref{eq:pushout}).
\[
\xymatrix@C=7pc @R=0.4pc{
&\mathbf{B}_{n-1}(X)\ar@(l,u)[lddd]_{Id}\ar@(r,u)[rddd]^{v_{1*}}\\
&\mbox{}&\\
&\mathbf{B}_{n-1}(X)\ar@/_/[ld]_-{Id}\ar@/^/[rd]^{v_{2*}}&\\
\mathbf{B}_{n-1}(X)& \mbox{} &\mathbf{B}_n(X)
\\
&\mathbf{B}_{n-1}(X)\ar@/^/[lu]^-{Id}\ar@/_/[ru]_{v_{3*}}\ar@{}[dd]|-{\vdots}&\\
&\mbox{}&\\
&\mathbf{B}_{n-1}(X)\ar@(l,d)[luuu]^{Id}\ar@(r,d)[ruuu]_{v_{k*}}
}
\]
Here each cycle involving $v_{1*}$ and $v_{i*}$ corresponds to the generator $t_i$, and we call $t_i$'s {\em stable letters} for $\mathbf{B}_n(X)$ as the ordinary HNN extension.
\end{rmk}

\section{\texorpdfstring{$k$}{k}-Connected sum of a pair of complexes}

We will use the generalized notion, called a {\em complex-of-groups}, to consider the braid group on $X\# Y$ of two given complexes $X$ and $Y$, and we briefly review about complex-of-groups. See \cite{Cor} for details.
\subsection{Complex-of-groups}
For two cells $\sigma$ and $\tau$ of a regular CW-complex, we denote by $\sigma\succ\tau$ if $\tau$ is a face of $\sigma$.
A face $\tau$ of $\sigma$ is {\em principal} if it is of codimension 1. By a {\em directed corner $\alpha$} of $\sigma$ we mean a triple $(\tau_1,\sigma,\tau_2)$ where $\tau_i$ are two different principal faces of $\sigma$ having a unique principal face $\tau_1\cap\tau_2$ in common. For a directed corner $\alpha=(\tau_1,\sigma,\tau_2)$, we denote by $\bar\alpha$ the inverse $(\tau_2,\sigma,\tau_1)$ of $\alpha$.

\begin{defn}[Complex-of-spaces]\label{def:complexofspaces}\cite{Cor}
A {\em (good) complex-of-spaces} $\mathcal{K}$ over $K$ is a CW-complex with a cellular map $p:\mathcal{K}\to K$ satisfying the conditions as follows:
\begin{enumerate}
\item for each cell $\sigma$ of $K$, there is a connected CW-complex $\mathcal{K}_\sigma$ with $p^{-1}(\sigma)\simeq \mathcal{K}_\sigma\times \sigma$;
\item for each cell $\sigma$ of $K$, the inclusion-induced map $\pi_1(\mathcal{K}_\sigma)\to\pi_1(\mathcal{K})$ is injective.
\end{enumerate}
\end{defn}

A $k$-skeleton $\mathcal{K}^{(k)}$ is defined as $p^{-1}(K^{(k)})$. Then the fundamental group $\pi_1(\mathcal{K})$ is isomorphic to $\pi_1(\mathcal{K}^{(2)})$. Moreover, there is a surjection $\pi_1(\mathcal{K}^{(1)})\to\pi_1(\mathcal{K})$ whose kernel is generated by elements corresponding to $\partial \tilde\sigma$ where $\tilde\sigma$ is a lift of 2-cell $\sigma\in K$ \cite[\S4]{Cor}.

\begin{defn}[Complex-of-groups]\label{def:complexofgroups}\cite{Cor}
A {\em complex-of-groups} $\mathcal{G}$ is a triple $(K, G, \phi)$ where
\begin{enumerate}
\item $K$ is a regular CW-complex;
\item $G$ assigns to each cell $\sigma$ of $K$ a group $G_\sigma$ and each pair $(\sigma,\tau)$ with $\sigma\succ\tau$ an injective homomorphism $i_{\sigma,\tau}:G_\sigma\to G_\tau$;
\item $\phi$ is a {\em corner labeling function} that assigns to each direct corner $\alpha=(\tau_1,\sigma,\tau_2)$ an element $\phi(\alpha)\in G_{\tau_1\cap\tau_2}$ satisfying the condition as follows:
\begin{enumerate}
\item $\phi(\bar\alpha)=\phi(\alpha)^{-1}$ for each directed corner $\alpha$;
\item If $\alpha=(\tau_1,\sigma,\tau_2)$, then the two compositions $G_\sigma\to G_{\tau_1}\to G_{\tau_1\cap\tau_2}$ and $G_\sigma\to G_{\tau_2}\to G_{\tau_1\cap\tau_2}$ differ by conjugation by $\phi(\alpha)$.
\end{enumerate}
\end{enumerate}
\end{defn}

For a complex-of-spaces $\mathcal{K}$ over $K$, an {\em associated} complex-of-groups $\mathcal{G}$ over $K$ can be defined by taking $G_\sigma=\pi_1(\mathcal{K}_\sigma)$. Note that for each $\sigma\succ\tau$, the inclusion $G_\sigma\to G_\tau$ depends on the choice of basepoints of $\mathcal{K}_\sigma$ and $\mathcal{K}_\tau$ and the choice of a path joining them. Hence, it is uniquely determined only up to inner automorphism on $G_\tau$.

A $k$-skeleton $\mathcal{G}^{(k)}$ of a complex-of-groups is nothing but a restriction on $K^{(k)}$. We say that $\mathcal{G}$ is a graph-of-groups when $K$ is 1-dimensional.

Let $\mathcal{G}$ be a graph-of-groups and $T$ be a maximal tree of $K$. We identify the generators for $\pi_1(\Gamma)$ with the set of {\em oriented} edges in $\Gamma\setminus T$. 
Let $\tilde\pi_1(\mathcal{G})$ be the free product of $\pi_1(K)$ and $\colim_T K$ of $K$ over $T$. 
Indeed, $\colim_T K$ is obtained from the {\em free product with amalgamation} of vertex groups along all edge groups.
Then we define a group $\pi_1(\mathcal{G})$ by {\em HNN extension} with all edge groups corresponding to the generators of $\pi_1(K)$.
More precisely, $\pi_1(\mathcal{G})$ is obtained by declaring $i_{e,v}(g)=e^{-1} i_{e,w}(g) e$ for all $g\in G_e$ in $\tilde\pi_1(G)$ for each edge $e=(v,w)\in \Gamma\setminus T$.

Similar to before, the fundamental group $\pi_1(\mathcal{G})$ is the same as $\pi_1(\mathcal{G}^{(2)})$ which is a quotient of $\pi_1(\mathcal{G}^{(1)})$ by elements coming from {\em corner and edge reading} for each 2-cell $\sigma$ of $K$.

For a 2-cell $\sigma$, let $\partial\sigma=e_1e_2\dots e_m$. Then the label $\phi(\sigma)$ on $\sigma$ is defined up to cyclic permutation as
\[
\phi(\sigma)=e_1\phi(\alpha_1)e_2\phi(\alpha_2)\dots e_m\phi(\alpha_m),
\]
where $\phi(\alpha_i)$ is a corner label for $\alpha_i=(e_i,\sigma,e_{i+1})$, and $e_i\in\pi_1(K^{(1)})$ is either trivial when it belongs to the maximal tree $T$ or a corresponding generator otherwise.

Remark that if a complex-of-groups $\mathcal{G}$ is associated with a complex-of-spaces $\mathcal{K}$, then $\pi_1(\mathcal{K})\simeq\pi_1(\mathcal{G})$, and so one may identify these two concepts only for the fundamental group.

\subsection{\texorpdfstring{$k$}{k}-connected sum}
Let $X$ and $Y$ be complexes, and $\vec v\in X$ and $\vec w\in Y$ be vertices of valency $k\ge 1$ with orderings $\lk(\vec v)=(v_1,\cdots,v_k)$, and $\lk(\vec w)=(w_1,\cdots, w_k)$, respectively.
We denote $(X,\vec v)\#(Y,\vec w)$ by $X\#Y$, edges $(v_i, w_i)$ by $e_i$ and $E=\cup e_i$ for simplicity.

For a given metric $d$ on $X\#Y$, by rescaling $X\#Y$ near $E$ after sufficient subdivisions, we may assume that for any configuration $\mathbf{x}\in B_n(X\#Y)$, the closure $\bar e_i$ of each $e_i$ contains at most 1 point of $\mathbf{x}$.
Then according to whether each $\bar e_i$ contains a point, $B_n(X\#Y)$ can be split as the disjoint union of subspaces indexed by (subset of) the power set of $E$ as follows.

Let $F\subset E$ be a subset of edges with $\#(F)=a\le n$, and let $\prod F$ denote the product of closures of edges contained in $F$, so that it is homeomorphic to a closed $a$-cube $D^a$. Then the spaces 
\[B_{r;s}(X_v;Y_w)\times \prod F\]
indexed by $F$ and $r,s$ with $r+s=n-a$ decompose $B_n(X\#Y)$ as desired.
We simply denote these pieces by $B_{r;s}(F)$ where $r+s=n-\#(F)$, and regard $B_{r;s}(F)$ as $\#(F)$-dimensional cube.
Then for each $e_i=(v_i,w_i)\in F$, the two maps defined by
\[i_{v_i}:B_{r;s}(F)\to B_{r+1;s}(F\setminus \{e_i\}),\quad
i_{w_i}:B_{r;s}(F)\to B_{r;s+1}(F\setminus \{e_i\})\]
correspond to two face maps of the $a$-dimensional cube $\prod F$. 
Moreover, by Proposition~\ref{prop:injection}, these induce injective homomorphisms on fundamental groups. We denote these maps by $v_{i*}$ and $w_{i*}$ for simplicity.

Hence all this information defines a complex-of-spaces $\mathcal{K}$ for $B_n(X\#Y)$ over the cube complex $K(n,k)$ depending on $n$ and $k$. Let $\mathcal{G}$ be an associated complex-of-groups with $\mathcal{K}$.
Then $\pi_1(\mathcal{K})=\mathbf{B}_n(X\#Y)=\pi_1(\mathcal{G})$.

\begin{rmk}
The dimension $\dim K$ of the cube complex $K$ is the minimum between $n$ and $k$, and moreover $K$ can be defined inductively as
\[K(n,k)=K(n,k-1)\sqcup K(n-1,k-1)\times I,\]
but this is not necessary in this paper and we omit the detail.
\end{rmk}

Since $\pi_1(\mathcal{G})$ depends only on the 2-skeleton of $K$ as mentioned before, we consider a 2-complex-of-groups $\mathcal{G}^{(2)}$ over the 2-skeleton $K^{(2)}$.

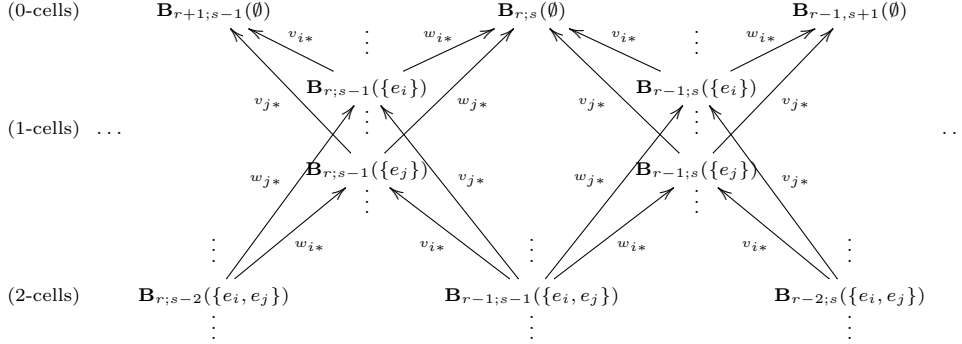
\begin{figure}[ht]
{\scriptsize\[
\xymatrix@C=0pc@R=-0.2pc{
\text{(0-cells)}& &\mathbf{B}_{r+1;s-1}(\emptyset)&& \mathbf{B}_{r;s}(\emptyset) & & \mathbf{B}_{r-1,s+1}(\emptyset)\\
& & &\vdots & & \vdots &\\
& & &\mathbf{B}_{r;s-1}(\{e_i\})\ar[ruu]^{w_{i*}}\ar[luu]_{v_{i*}} & & \mathbf{B}_{r-1;s}(\{e_i\})\ar[luu]_{v_{i*}} \ar[ruu]^{w_{i*}}\\
\text{(1-cells)}& \dots& &\vdots & & \vdots & &\dots\\
& & &\mathbf{B}_{r;s-1}(\{e_j\})\ar[ruuuu]_{w_{j*}}\ar[luuuu]^{v_{j*}} & & \mathbf{B}_{r-1;s}(\{e_j\})\ar[ruuuu]_{v_{j*}}\ar[luuuu]^{v_{j*}}\\
& & & \vdots & & \vdots &\\
& & \vdots & & \vdots & & \vdots\\
\text{(2-cells)}& &\mathbf{B}_{r;s-2}(\{e_i, e_j\})\ar[ruuuuu]^{w_{j*}}\ar[ruuu]_{w_{i*}}& &\mathbf{B}_{r-1;s-1}(\{e_i, e_j\})\ar[luuuuu]_{v_{j*}}\ar[luuu]^{v_{i*}}\ar[ruuuuu]^{w_{j*}}\ar[ruuu]_{w_{i*}} & & \mathbf{B}_{r-2;s}(\{e_i,e_j\})\ar[luuuuu]_{v_{j*}}\ar[luuu]^{v_{i*}}\\
& &\vdots & & \vdots & & \vdots
}
\]}
\caption{A complex-of-groups $\mathcal{G}$}
\label{fig:complexofgroups}
\end{figure}

\begin{figure}[ht]
{\scriptsize\[
\xymatrix@C=1pc{
\{v_i, v_j\}& e_i\cup \{v_j\}\ar[l]_{i_{v_i}}\ar[r]^{i_{w_i}} & \{w_i, v_j\}\\
\{v_i\}\cup e_j\ar[u]^{i_{v_j}}\ar[d]_{i_{w_j}}& e_i\times e_j\ar[l]_{i_{v_i}}\ar[r]^{i_{w_i}}\ar[u]^{i_{v_j}}\ar[d]_{i_{w_j}} & \{w_i\} \cup e_j\ar[u]^{i_{v_j}}\ar[d]_{i_{w_j}}\\
\{v_i, w_j\}& e_i\cup \{w_j\}\ar[l]_{i_{v_i}}\ar[r]^{i_{w_i}} & \{w_i, w_j\}\\
}\qquad
\xymatrix@C=1pc{
\mathbf{B}_{r+1;s-1}(\emptyset)& \mathbf{B}_{r;s-1}(\{e_i\})\ar[l]_{v_{i*}}\ar[r]^{w_{i*}}&\mathbf{B}_{r;s}(\emptyset)\\
\mathbf{B}_{r;s-1}(\{e_j\})\ar[u]^{v_{j*}}\ar[d]_{w_{j*}}& 
\mathbf{B}_{r-1;s-1}(\{e_i,e_j\})\ar[u]^{v_{j*}}\ar[d]_{w_{j*}}\ar[l]_{v_{i*}}\ar[r]^{w_{i*}}
&\mathbf{B}_{r-1;s}(\{e_j\})\ar[u]^{v_{j*}}\ar[d]_{w_{j*}}\\
\mathbf{B}_{r;s}(\emptyset)& \mathbf{B}_{r-1;s}(\{e_i\})\ar[l]_{v_{i*}}\ar[r]^{w_{i*}}&\mathbf{B}_{r-1;s+1}(\emptyset)
}
\]
}
\caption{A 2-cell $e_i\times e_j$ in $K$ and corresponding complex-of-groups $\mathbf{B}_{r-1;s-1}(\{e_i,e_j\})$ in $\mathcal{K}$}
\label{fig:2cell}
\end{figure}

The way to compute $\pi_1(\mathcal{G}^{(2)})$ is described earlier, and
before doing the computation, we fix basepoints $*_r$ and $*_s$ for $B_r(X_v)$ and $B_s(Y_w)$ for each $r$ and $s$, respectively. We denote $*_r\sqcup *_s$ by $*_{r;s}$.

Let $F\subset E$ with $\#(F)\le n$ and $e_1\not\in F$.
For each $r$, we glue $B_{r;s}(F)$ and $B_{r-1;s+1}(F)$ via $B_{r-1;s}(F\cup\{e_1\})$ by using paths $e_1^{r-1;s}$ from $*_{r;s}$ to $*_{r-1;s+1}$.
Then the distinction on basepoints $*_{r;s}$ for $B_{r;s}(F)$ disappears, and we denote them by $*_m$ if $m=r+s$.
Moreover, we may assume that all points of each $*_m$ are lying on $e_1$, and label them by $\{1,\dots,m\}$ with respect to the order coming from the orientation $v_1\to w_1$.

For $i\ge 2$, we choose paths $\gamma_i^m$ and $\delta_i^m$ from $*_m$ to $i_{v_i}(*_{m-1})$ and $i_{w_i}(*_{m-1})$ such that $\gamma_i^m$ moves only the first point in $X_v$, and $\delta_i^m$ moves only the last point in $Y_w$.
For convenience's sake, we set $\gamma_1^m$ and $\delta_1^m$ to be constant paths. Notice that $e_1$ also defines a constant path but has the effect of changing the domain from $B_{r;s}(F)$ to $B_{r-1;s+1}(F)$. We will suppress the decorations for $\gamma_i$ and $\delta_i$ unless any ambiguity occurs.
Note that two maps $v_{i*}$ and $w_{i*}$ are precisely $(\gamma_i)_*^{-1} (i_{v_i})_*$ and $(\delta_i)_*^{-1}(i_{w_i})_*$, respectively.

\begin{lem}
Let $(X, v)$ and $(Y,w)$ be as above. Then there exist homomorphisms
\[\Phi_X:\mathbf{B}_n(X)\to\mathbf{B}_n(X\# Y),\quad\Phi_Y:\mathbf{B}_n(Y)\to\mathbf{B}_n(X\# Y).\]
\end{lem}
\begin{proof}
Since $X_v$ and $Y_w$ are connected, there are embedded trees $T_X$ and $T_Y$ in $X_v$ and $Y_w$, respectively, such that $\partial T_X=\ln(v)$ and $\partial T_Y=\ln(w)$.
We denote by $q_X:T_Y\to T_k$ and $q_Y:T_X\to T_k$ the label-preserving map defined in Example~\ref{ex:tree}.
Hence as described in Lemma~\ref{lem:thetasubgroup}, we have 
\[\hat q_X: X_v\cup T_Y \to X_v\cup T_k = X \text{ and }\hat q_Y: T_X\cup Y_w \to T_k\cup Y_w = Y,\]
which induce $\hat q_X^*$ and $\hat q_Y^*$ by Proposition~\ref{prop:generaledgecontraction}.
Hence we have the desired homomorphisms 
\[
\Phi_X:\mathbf{B}_n(X)\stackrel{\hat q_X^*}{\longrightarrow}\mathbf{B}_n(X_v\cup T_Y)\to\mathbf{B}_n(X\# Y)
\]
and
\[
\Phi_Y:\mathbf{B}_n(Y)\stackrel{\hat q_Y^*}{\longrightarrow}\mathbf{B}_n(T_X\cup Y_w)\to\mathbf{B}_n(X\# Y)
\]
by composing the maps induced by the inclusions $X_v\cup T_Y\to X\# Y$ and $T_X\cup Y_w\to X\#Y$.
\end{proof}

Now we can do exactly the same business as before. Let $t_{i,r}=[\gamma_i e_i^{r-1;n-r}\delta_i^{-1}]$ for $1\le r\le n$. More precisely, the action of $t_{i,r}$ on $\mathbf{B}_{r-1;s}(X_v;Y_w)$ is as follows.
\begin{equation}\label{eq:edge}
t_{i,r}^{-1} v_{i*}(\beta) t_{i,r} = w_{i*}(\beta),
\end{equation}
where for all $\beta\in\pi_1({B}_{r-1;s}(\{e_i\}))\simeq\mathbf{B}_{r-1;s}(X_v;Y_w)$ with $r+s=n$.

Therefore, 
\begin{align*}
\pi_1(\mathcal{G}^{(1)})&=\frac{\left(\mathop{\ast}_{r+s=n}\mathbf{B}_{r;s}(X_v;Y_w)\right)\ast\langle t_{i,r}|2\le i\le k, 1\le r\le n\rangle}
{\langle\!\langle
t_{i,r}^{-1} v_{i*}(\beta) t_{i,r} = w_{i*}(\beta)\quad \beta\in\mathbf{B}_{r-1;s}(X_v;Y_w)
\rangle\!\rangle}
\end{align*}

As before, we may identify $\mathbf{B}_r(X_v)$ and $\mathbf{B}_s(Y_w)$ with subgroups of $\mathbf{B}_n(X_v)$ and $\mathbf{B}_n(Y_w)$ via $v_{1*}$ and $w_{1*}$, respectively.
Then for $\alpha\in\mathbf{B}_{r-1}(X_v)$, the supports of $\alpha$ and $\delta_i$ are disjoint. Hence $w_{i*}(\alpha)=\alpha$, and similarly, $v_{i*}(\beta)=\beta$ for any $\beta\in\mathbf{B}_s(Y_w)$. That is,
\begin{align*}
\pi_1(\mathcal{G}^{(1)})&=\frac{\left(\mathop{\ast}_{r+s=n}\mathbf{B}_{r;s}(X_v;Y_w)\right)\ast\langle t_{i,r}|2\le i\le k, 1\le r\le n\rangle}
{\left\langle\!\!\!\left\langle
\begin{cases}
t_{i,r}^{-1} v_{i*}(\alpha) t_{i,r} = \alpha\quad \alpha\in\mathbf{B}_{r-1}(X_v)\\
t_{i,r}^{-1} \beta t_{i,r} = w_{i*}(\beta)\quad \beta\in\mathbf{B}_s(Y_w)
\end{cases}
\right\rangle\!\!\!\right\rangle}
\end{align*}

Moreover, this is a quotient of 
\[
\mathbf{B}_n(X_v)\ast\mathbf{B}_n(Y_w)\ast\left\langle t_{i,r}, u_{i,r}\left| u_{i,r}=t_{i,n-r}^{-1},2\le i\le k, 1\le r\le n\right.\right\rangle
\]
since $\mathbf{B}_{r;s}(X_v;Y_w)=\mathbf{B}_r(X_v)\times\mathbf{B}_s(Y_w)$ and both $\mathbf{B}_r(X_v)$ and $\mathbf{B}_s(Y_w)$ are subgroups of $\mathbf{B}_n(X_v)$ and $\mathbf{B}_n(Y_w)$, respectively.
Here, we furthermore add a set $\{u_{i,r}\}$ of dummy generators by declaring $u_{i,r}=t_{i,n-r}^{-1}$. 
Then it becomes
\begin{align*}
\pi_1(\mathcal{G}^{(1)})&=\frac{\mathbf{B}_n(X_v)\ast\mathbf{B}_n(Y_w)\ast\left\langle t_{i,r}, u_{i,r}\left| u_{i,r}=t_{i,n-r}^{-1},2\le i\le k, 1\le r\le n\right.\right\rangle}
{\left\langle\!\!\!\left\langle
\begin{cases}
[\mathbf{B}_r(X_v),\mathbf{B}_s(Y_w)]& r+s=n\\
t_{i,r}^{-1} v_{i*}(\alpha) t_{i,r} = \alpha & \alpha\in\mathbf{B}_{r-1}(X_v)\\
u_{i,s}^{-1} w_{i*}(\beta) u_{i,s} = \beta & \beta\in\mathbf{B}_{s}(Y_w)
\end{cases}
\right\rangle\!\!\!\right\rangle}.
\end{align*}

Notice that the second and third types of defining relators are those appearing in $\mathbf{B}_n(X)$ and $\mathbf{B}_n(Y)$.

Suppose $k=1$. Then $X_v\equiv_B X$, $Y_w\equiv_B Y$ and $X\# Y$ is just a boundary wedge sum $X\vee_{\partial} Y$ of $X$ and $Y$.
Since there is no $t_{i,r}$, we have 
\[
\mathbf{B}_n(X\vee_{\partial} Y)\simeq
\frac{\mathbf{B}_n(X)\ast\mathbf{B}_n(Y)}
{\left\langle\!\left\langle
[\mathbf{B}_r(X),\mathbf{B}_s(Y)], r+s=n
\right\rangle\!\right\rangle}.
\]
In this case, we have a graph-of groups as well over the linear graph of length $n$. Hence $\mathbf{B}_n(X\vee_\partial Y)$ is obtained by the iterated amalgamated free product.

On the other hand, if $n=1$, then the decoration for $t_i$ is not necessary. Therefore
\[
\pi_1(X\#Y)\simeq
\pi_1(X_v)\ast\pi_1(Y_w)\ast\langle t_2,\dots, t_k\rangle\simeq
\frac{\pi_1(X)\ast\pi_1(Y)}
{\langle\!\langle t_i u_i | 2\le i\le k
\rangle\!\rangle},
\]
where $t_i$'s and $u_i$'s in the right correspond to generators defined as before. This is nothing but the usual Seifert-van Kampen theorem.

Suppose $k, n\ge 2$ and let $F_{i,j}=\{e_i,e_j\}$ with $i<j$.
Then the boundary $\partial\left(\prod F\right)$ has 4 corners, and so we have to consider 8 maps, $\{LU,UL, UR, RU, RD, DR, DL, LD\}$ corresponding to the ways of compositions  as shown in Figure~\ref{fig:2cell}.

In the northwest corner, we need consider two maps left-and-up $LU$ and up-and-left $UL$.
Then for $r+s=n-2$, the maps $LU$ and $UL$ are the compositions
\[
LU:\mathbf{B}_{r;s}(X_v;Y_w)\xrightarrow{v_{i*}}\mathbf{B}_{r+1;s}(X_v;Y_w)\xrightarrow{v_{j*}}\mathbf{B}_{r+2;s}(X_v;Y_w)
\]
\[
UL:\mathbf{B}_{r;s}(X_v;Y_w)\xrightarrow{v_{j*}}\mathbf{B}_{r+1;s}(X_v;Y_w)\xrightarrow{v_{i*}}\mathbf{B}_{r+2;s}(X_v;Y_w),
\]
which are nothing but conjugates by two paths $\eta_1$ and $\eta_2$ from $*_n$ to $*_{n-2}\sqcup\{v_i,v_j\}$, 
where $\eta_1$ moves the first point to $v_j$ and the second point to $v_i$ but $\eta_2$ moves the first to $v_i$ and the second to $v_j$.
More precisely, $\eta_1=\gamma_j\cdot i_{v_j}(\gamma_i)$ and $\eta_2=\gamma_i\cdot i_{v_i}(\gamma_j)$, and therefore they differ by the loop $\eta_2\cdot\eta_1^{-1}$ which represents the element exactly the same as $s_{i,j}$ defined  in Example~\ref{ex:tree}, which is the image of the generator of $\mathbf{B}_2(T_3)$ via the map induced from the embedding $(T_3,(1,2,3))\to (T_X,(1,i,j))$.
Note that when $i=1$, then we set $s_{1,j}$ to be trivial for all $1<j$.

In summary, 
\[
UL(\cdot) = s_{i,j} LU(\cdot) s_{i,j}^{-1}.
\]

In the southeast corner, we have a similar result as above. That is, the two maps $RD$ for right-and-down and $DR$ for down-and-right are related as
\[
DR(\cdot)=s_{i,j}' RD(\cdot) s_{i,j}'^{-1},
\]
where $s_{i,j}'$ is the image of the generator of $\mathbf{B}_2(T_3)$ via the map induced from the embedding $(T_3,(1,2,3))\to(T_Y,(1,i,j))$, and is set to be trivial if $i=1$.

In the northeast corner, two maps $UR$ for up-and-right and $RU$ for right-and-up coincide since the supports of $\gamma_j$ and $\delta_i$ are disjoint, and so we need not conjugate for transport.

In the southwest corner, this is the exactly same situation as above and so the two maps $LD$ and $DL$ for left-and-down and down-and-left coincide.

On the other hand, the four edges of $\partial(\prod F)$ are as follows.
\begin{align*}
U: v_{i*}(\beta) \mapsto w_{i*}(\beta),\qquad &L: w_{j*}(\beta) \to v_{j*}(\beta)\qquad\forall\beta\in\mathbf{B}_{r+1;s}(X_v;Y_w),\\
R: v_{j*}(\beta) \mapsto w_{j*}(\beta),\qquad &D: w_{i*}(\beta) \to v_{i*}(\beta)\qquad\forall\beta\in\mathbf{B}_{r;s+1}(X_v;Y_w).
\end{align*}

Then as seen in the computation of $\pi_1(\mathcal{G}^{(1)})$, the maps $U,R,D$ and $L$ satisfy relations similar to (\ref{eq:edge}), which are given by
\begin{align*}
U(\cdot)=t_{i,r+2}^{-1} (\cdot) t_{i,r+2}, &\quad
L(\cdot)=u_{j,s}^{-1}(\cdot) u_{j,s}, \\
R(\cdot)=t_{j,r+1}^{-1} (\cdot) t_{j,r+1}, &\quad
D(\cdot)=u_{i,s+1}^{-1}(\cdot) u_{i,s+1}.
\end{align*}
Recall that we set $t_{1,r}=u_{1,s}$ to be trivial.

Finally, the reading of the edges and corners of $\partial[e_i,e_j]$ gives us a word
\[
H_{i,j} = s_{i,j}^{-1} t_{i,r+2} t_{j,r+1} s_{i,j}'^{-1} u_{i,s+1} u_{j,s},
\]
and $\pi_1(\mathcal{G}^{(2)})$ is obtained by declaring $H_{i,j}=e$ for all $1\le i<j\le k$.

Suppose $i=1$. Then since $r+s=n-2$ with $0\le r\le n-2$,
\[
H_{1,j}=t_{j,r+1}u_{j,s}= t_{j,r+1} t_{j,r+2}^{-1}.
\]
Therefore the defining relator $H_{1,j}=e$ removes all decorations on $t_j$'s and on $u_j$'s as well.

If $i>1$, then $H_{i,j}=e$ gives 
$s_{i,j}^{-1} t_i t_j s_{i,j}'^{-1} u_i u_j=e$, or equivalently, 
\[
s_{i,j}t_j t_i s_{i,j}' u_j u_i=e.
\]

In summary, 
\begin{equation}\label{eq:connectedsum}
\mathbf{B}_n(X\# Y)
=\frac{\mathbf{B}_n(X)\ast\mathbf{B}_n(Y)}
{\left\langle\!\!\!\left\langle
\begin{cases}
[\mathbf{B}_r(X_v),\mathbf{B}_s(Y_w)] & r+s=n\\
t_i  u_i = e& 2\le i\le k\\
s_{i,j} t_j t_i s_{i,j}' u_j u_i = e& 2\le i<j\le k
\end{cases}
\right\rangle\!\!\!\right\rangle}.
\end{equation}

Before we state the theorem, we will discuss the $s_{i,j}$'s further.
As mentioned above, both $s_{i,j}$ and $s_{i,j}'$ naturally come from $\mathbf{B}_2(T_k)$ as follows.

Let us break $\Theta_k$ into $T_k^L$ and $T_k^R$, which are left and right halves. That is, we may assume that if $\br(\Theta_k)=\{0,0'\}$,
\[T_k^L=\Theta_k\setminus \st(0),\quad T_k^R=\Theta_k\setminus \st(0').\]

We also denote the closures of these halves by $\Theta_k^L$ and $\Theta_k^R$, and denote the maps by $\widehat{(\cdot)}_{L}$ and $\widehat{(\cdot)}_{R}$.
\[\widehat{(\cdot)}_L:T_k^L\to\Theta_k^L,\quad\widehat{(\cdot)}_R:T_k^R\to\Theta_k^R.\]
Then as seen in Example~\ref{ex:theta}, 
\[\mathbf{B}_2(\Theta_k^L)=\langle \sigma_{i,j}, t_\ell\rangle,\quad
\mathbf{B}_2(\Theta_k^R)=\langle \sigma'_{i,j}, u_\ell\rangle.\]

Since $\Theta_k^L\#\Theta_k^R = \Theta_k$, we can use (\ref{eq:connectedsum}) to compute $\mathbf{B}_n(\Theta_k)$.
Notice that $s_{i,j}$ and $s_{i,j}'$ correspond to $\sigma_{i,j}$ and $\sigma_{i,j}'$, respectively since they essentially come from $\mathbf{B}_2(T_k^L)$ and $\mathbf{B}_2(T_k^R)$ by definition.
Therefore,
\[\mathbf{B}_2(\Theta_k)=\mathbf{B}_2(\Theta_k^1\# \Theta_k^2)=\langle \sigma_{i,j}, t_\ell, \sigma'_{i,j}, u_\ell |
t_\ell u_\ell =e, \sigma_{i,j}t_j t_i \sigma'_{i,j}  u_j u_i=e\rangle,\]
which is obviously isomorphic to $\mathbf{B}_2(\Theta_k^L)$ and $\mathbf{B}_2(\Theta_k^R)$.

Now we turn back to $\mathbf{B}_n(X\# Y)$.
Recall the surjective map $\xi$ described in Example~\ref{ex:theta}. In this situation, we have two surjections 
\[\xi_L:\mathbf{B}_2(\Theta_k)\to\mathbf{B}_n(\Theta_k^L),\quad
\xi_R:\mathbf{B}_2(\Theta_k)\to\mathbf{B}_n(\Theta_k^R)\]
satisfying that 
\[\widehat{(\cdot)}_{L,*}\circ(i_{v_1})_*^{n-2}=\xi_L\circ \widehat{(\cdot)}_{L,*},\quad
\widehat{(\cdot)}_{R,*}\circ(i_{w_1})_*^{n-2}=\xi_R\circ \widehat{(\cdot)}_{R,*}.\]

Then $s_{i,j}$ and $s_{i,j}'$ are nothing but the images of $\sigma_{i,j}$ and $\sigma_{i,j}'$ in $\mathbf{B}_2(\Theta_k)$ under the compositions $\Psi_X\circ\xi_L$ and $\Psi_Y\circ\xi_R$.
\[(\Psi_X\circ\xi_L)(\sigma_{i,j})=s_{i,j}\in\mathbf{B}_n(X),\quad
(\Psi_Y\circ\xi_R)(\sigma_{i,j}')=s_{i,j}'\in\mathbf{B}_n(Y).\]

Moreover, we have
\[(\Psi_X\circ\xi_L)(t_\ell)=t_\ell\in\mathbf{B}_n(X),\quad
(\Psi_Y\circ\xi_R)(u_\ell)=u_\ell\in\mathbf{B}_n(Y),\]
which correspond to the stable letters for $\mathbf{B}_n(X)$ and $\mathbf{B}_n(Y)$.

Therefore, the second and third types of defining relations in (\ref{eq:connectedsum}) precisely declare that the two images of $\mathbf{B}_2(\Theta_k)$ under $\Psi_X\circ\xi_L$ and $\Psi_Y\circ\xi_R$ are the same.
Moreover, both $\Psi_X\circ\xi_L$ and $\Psi_Y\circ\xi_R$ factor through $\xi:\mathbf{B}_2(\Theta_k)\to\mathbf{B}_n(\Theta_k)$, 
so there exist $\tilde\Psi_X$ and $\tilde\Psi_Y$ satisfying the following commutative diagram, where the innermost square involving $\tilde\Psi_X$ and $\tilde\Psi_Y$ is the push-out diagram.
\[
\xymatrix{
 &\mathbf{B}_n(\Theta_k)\ar[r]^{\Psi_X} & \mathbf{B}_n(X)\ar[rd]\ar[rrd]^{\Phi_X}\\
\mathbf{B}_2(\Theta_k)\ar@{->>}[ru]^{\xi_L}\ar@{->>}[rd]_{\xi_R}\ar@{->>}[r]^{\xi} & \mathbf{B}_n(\Theta_k)\ar[u]_{\simeq}\ar[d]^{\simeq} \ar[ru]_{\tilde\Psi_X}\ar[rd]^{\tilde\Psi_Y} & & \widetilde{\mathbf{B}}_n(X;Y)\ar@{-->}[r]^-{\exists Q}&\mathbf{B}_n(X\#Y)\\
 &\mathbf{B}_n(\Theta_k)\ar[r]_{\Psi_Y} & \mathbf{B}_n(Y)\ar[ru]\ar[rru]_{\Phi_Y}
}
\]

Hence the group $\widetilde{\mathbf{B}}_n(X;Y)$ is defined as the free product with amalgamation as follows.
\begin{align*}
\widetilde{\mathbf{B}}_n(X;Y) &= \frac{
\mathbf{B}_n(X)\ast\mathbf{B}_n(Y)}
{\langle\!\langle
\tilde\Psi_X(\beta)=\tilde\Psi_Y(\beta)|\beta\in\mathbf{B}_n(\Theta_k)
\rangle\!\rangle}\\
&=\frac{
\mathbf{B}_n(X)\ast\mathbf{B}_n(Y)}
{\left\langle\!\!\!\left\langle
\begin{cases}
t_i  u_i = e& 2\le i\le k\\
s_{i,j} t_j t_i s_{i,j}' u_j u_i = e& 2\le i< j\le k
\end{cases}
\right\rangle\!\!\!\right\rangle}.
\end{align*}

Finally, the map $Q$ is obviously taking the quotient $\widetilde{\mathbf{B}}_n(X;Y)$ by
\[\langle\!\langle[\mathbf{B}_r(X_v),\mathbf{B}_s(X_w)], r+s=n\rangle\!\rangle.\]

In summary, we have the following theorem.
\begin{thm}\label{thm:connectedsum}
Let $X, Y$ be complexes, $\vec v\in X$ and $\vec w\in Y$ be vertices of valency $k\ge 1$ with orderings $\lk(\vec v)=(v_1,\dots,v_k)$ and $\lk(\vec w)=(w_1,\dots,w_k)$, respectively. Then the braid group $\mathbf{B}_n((X,\vec v)\# (Y,\vec w))$ for $n\ge 1$ is as follows.
\begin{align*}
\mathbf{B}_n((X,\vec v)\# (Y,\vec w))
=\mathbf{B}_n(X)\ast_{\mathbf{B}_n(\Theta_k)}\mathbf{B}_n(X)
\big/\langle\!\langle[\mathbf{B}_r(X_v),\mathbf{B}_s(Y_w)], r+s=n\rangle\!\rangle
%\\
%=\frac{\mathbf{B}_n(X)\ast\mathbf{B}_n(Y)}
%{\left\langle\!\!\!\left\langle
%\begin{cases}
%[\mathbf{B}_r(X_v),\mathbf{B}_s(Y_w)] & r+s=n\\
%\tilde\Psi_X(\beta)=\tilde\Psi_Y(\beta)&\beta\in\mathbf{B}_n(\Theta_k)
%\end{cases}
%\right\rangle\!\!\!\right\rangle},
\end{align*}
for
\[
\mathbf{B}_n(X)\xleftarrow{\tilde\Psi_X}\mathbf{B}_n(\Theta_k)\xrightarrow{\tilde\Psi_Y}\mathbf{B}_n(Y),
\]
where $\tilde\Psi_X$ and $\tilde\Psi_Y$ are defined as 
\[\tilde\Psi_X\circ\xi=\Psi_X\circ\xi_L,\quad\tilde\Psi_Y\circ\xi=\Psi_Y\circ\xi_R.\]
\end{thm}

\begin{ex}[2-braid group on the union of two trees]\label{ex:twotrees}
Let $T$ and $T'$ be trees with $k=\#(\partial T)=\#(\partial T')$, and $\widehat T$ and $\widehat T'$ be $k$-closures whose closing vertices are denoted by $v$ and $w$, respectively.
We fix orderings on $\lk(v)$ and $\lk(w)$, and consider 2-braid group on $(\widehat T,\vec v)\#(\widehat T',\vec w)$.
Then by Example~\ref{ex:treeclosure} and Theorem~\ref{thm:connectedsum}, 
\[
\mathbf{B}_2(\widehat T\#\widehat T')=\left\langle
s_{i,j},s'_{i,j},t_r\left|
\begin{matrix}
s_{i,j}' =t_i^{-1}t_j^{-1}s_{i,j}^{-1}t_i t_j& 2\le i<j\le k\\
s_{i,j}=s_{i',j'}&(i,j)\sim_T(i',j')\\
s_{i,j}'=
s_{i',j'}' & (i,j)\sim_{T'}(i',j')
\end{matrix}
\right.
\right\rangle.
\]

Note that the generators $s_{i,j}'$ are not necessary by the first type of defining relators, moreover, the third reduces $s_{i,j}$'s as well. Indeed, under certain condition it has a generating set consisting of $t_r$'s and only one $s_{i,j}$.
We will see this later.
\end{ex}

\begin{rmk}
Both $k$-closures and $k$-connected sums for any $k\ge 1$ can be obtained by using iterated $2$-closures and $1$-connected sums, which always give graph-of-groups structures.
\end{rmk}

\section{Applications -- Embeddabilities and connectivities}

In this section, we will prove Theorem~\ref{thm:emb}~(2) and (3), namely, how the braid group $\mathbf{B}_n(X)$, or its abelianization $H_1(\mathbf{B}_n(X))$ is related with the geometry of $X$.
Unless mentioned otherwise, we assume that $X$ is sufficiently subdivided and simple.

\subsection{Surface embeddability}
We start with the following easy observation.
\begin{lem}\label{lem:embed}
Let $X$ and $Y$ be simple complexes. Then $X$ and $Y$ embed into surfaces if and only if so do $\widehat X$ and $\widehat Y$ if and only if so does $X\# Y$,
for arbitrary closures and connected sums.

Especially, $X$ embeds into a surface if and only if so does any elementary subcomplex of $X$.
\end{lem}

\begin{lem}\label{lem:torsioncombination}
Let $X$ and $Y$ be complexes.
Then the following are equivalent.
\begin{enumerate}
\item $\mathbf{B}_n(X)$ and $\mathbf{B}_n(Y)$ are torsion-free for all $n\ge 1$.
\item $\mathbf{B}_n(\widehat X)$ and $\mathbf{B}_n(\widehat Y)$ are torsion-free for all $n\ge 1$.
\item $\mathbf{B}_n(X\# Y)$ is torsion-free for any $n\ge 1$.
\end{enumerate}
\end{lem}
\begin{proof}
As remarked in the end of the previous section, both $k$-closures and $k$-connected sums yield graph-of-groups structures, which correspond to HNN-extensions and free products with amalgamations.
The proof follows since these group operations preserve both torsions and torsion-freeness.
\end{proof}

It is worth remarking that indeed the configuration space $B_n(X)$ or $B_n(X\# Y)$ is {\em aspherical} if and only if so is $B_n(X_v)$ or so are $B_n(X)$ and $B_n(Y)$, respectively. This follows easily by considering graphs-of-spaces structures on configuration spaces.

\begin{cor}\label{cor:torsion}
Let $X$ be a complex, not necessarily elementary. Suppose $X$ can not be embedded into any surface. Then $\mathbf{B}_n(X)$ contains $\mathbf{S}_n$ as a subgroup.
\end{cor}
\begin{proof}
By Lemma~\ref{lem:embed}, there exists an elementary subcomplex $Y$ in $X$, which does not embed into any surface.
Then as mentioned in Example~\ref{ex:S_0}, we suppose that there exists an embedding $i:S_0\to Y\subset X$.

Let $\rho$ and $\rho'$ be the induced permutations from $\mathbf{B}_n(S_0)$ and $\mathbf{B}_n(X)$, respectively.
Then it is obvious that $\rho=\rho'\circ i_*$. However, $\rho$ is an isomorphism and therefore $\rho'\circ i_*\circ\rho^{-1}$ is the identity on $\mathbf{S}_n$. In other words,
\[i_*\circ\rho^{-1}:\mathbf{S}_n\to\mathbf{B}_n(X)\]
is injective.
\end{proof}

\begin{proof}[Proof of Theorem~\ref{thm:emb}~(2)]
Suppose $X$ can be embedded into a surface $\Sigma$. Then it can be modified to a simple complex $X'$ by using only the reverse process of edge contractions as Proposition~\ref{prop:graphedgecontraction}.
Hence $X'\equiv_B X$ and embeds into $\Sigma$, and moreover, all the elementary complexes contained in $X'$ embed into surfaces by Lemma~\ref{lem:embed}, as well.
Therefore by Lemma~\ref{lem:elementary}, all their braid groups are torsion-free.

As seen in Lemma~\ref{lem:torsioncombination}, since the build-up processes preserve torsion-freeness, $\mathbf{B}_n(X)$ is torsion-free, as desired.

The converse follows from Corollary~\ref{cor:torsion}.
\end{proof}

\subsection{The first homology groups}
Suppose $n\ge 2$. Then the induced permutation $\rho$ induces the map 
\[\bar\rho:H_1(\mathbf{B}_n(X))\to H_1(\mathbf{S}_n)\simeq\mathbb{Z}_2.\]
\begin{cor}\label{cor:elementary}
Let $X$ be an elementary complex.
Then for $n\ge 2$,
\[
H_1(\mathbf{B}_n(X))\simeq\begin{cases}
\mathbb{Z}^r& X=T_k;\\
H_1(X)\oplus\langle[\sigma]\rangle & \dim X=2, X\text{ is planar};\\
H_1(X)\oplus\langle[\sigma]\rangle/\langle 2[\sigma]\rangle & otherwise,
\end{cases}
\]
where $\bar\rho([\sigma])$ is nontrivial in $H_1(\mathbf{S}_n)\simeq \mathbb{Z}_2$, and $r=r(n,k,k)$ is given by the equation $(\ref{eq:rank})$.

Especially, $H_1(\mathbf{B}_n(X))$ is torsion-free if and only if $X$ is planar.
\end{cor}
\begin{proof}
This is a direct consequence of the discussion in Section~\ref{sec:elementary}.
\end{proof}

Hence Theorem~\ref{thm:emb}~(3) holds for elementary complexes, and from now on 
we assume that $X$ is nonelementary, and furthermore $\partial X\neq\emptyset$. If $\partial X=\emptyset$, then there exists $w\in \br(X)$ such that $S_0\subset\st(w)$, and we obtain a boundary by attaching a disk in $\st(w)$ as depicted in Figure~\ref{fig:S_0'}.

Therefore we can consider a map $\mathbf{B}_1(X)\to\mathbf{B}_n(X)$ which is a composition of $(n-1)$ maps described in Proposition~\ref{prop:injection}.
This induces the map $i_{X,n}:H_1(X)\to H_1(\mathbf{B}_n(X))$, whose cokernel plays an important role in proving our theorem.

For example, since $\bar\rho$ is a surjection and $H_1(X)\subset\ker(\bar\rho)$, we have a surjection $\coker(i_{X,n})\to H_1(\mathbf{S}_n)\simeq\mathbb{Z}_2$. Therefore $\coker(i_{X,n})$ can never be trivial for any $n\ge 2$.
On the other hand, if we have an embedding $\iota:X\to Y$, then it induces a map $\iota_*:\coker(i_{X,n})\to\coker(i_{Y,n})$. Since $\bar\rho$ is equivariant, $\iota_*$ is nontrivial too.
\begin{cor}\cite{KKP}\label{cor:nonplanargraph}
Suppose $X$ contains a nonplanar graph. Then $\coker(i_{X,n})$ has $2$-torsion.
\end{cor}
\begin{proof}
For a nonplanar graph $\Gamma$, it is known that $H_1(\mathbf{B}_n(\Gamma))$ has 2-torsion corresponding to the generator for $H_1(\mathbf{S}_n)$. See \cite{KKP} or Proposition~\ref{prop:connectedsum} below. Hence for any complex $X$ containing a nonplanar graph, $\coker(i_{X,n})$ has 2-torsion as discussed above.
\end{proof}

We observe how $k$-closure and $k$-connected sum affect the first homology of braid groups as follows. These two lemma are direct consequences of Theorem~\ref{thm:closure} and Theorem~\ref{thm:connectedsum}.

\begin{lem}\label{lem:exactseq1}
Let $X$ be a complex and $\widehat X$ be a $k$-closure of $X$. Then there exists a commutative diagram with exact rows as follows.
\[
\xymatrix{
1\ar[r]&H_1(X)\ar[r]\ar[d]^{i_{X,n}}&
H_1(\widehat X)\ar[r]\ar[d]^{i_{\widehat X,n}}&
H_1(\Theta_k)\ar[r]\ar[d]^{\simeq}&1\\
&H_1(\mathbf{B}_n(X))\ar[r]&
H_1(\mathbf{B}_n(\widehat X))\ar[r]&
H_1(\Theta_k)\ar[r]&1.
}
\]
\end{lem}

Note that since $H_1(\Theta_k)\simeq\mathbb{Z}^{k-1}$ is free abelian, the surjective map in each row splits.

\begin{cor}\label{cor:closures}
Let $Y$ be an elementary complex of dimension 2. Suppose $X$ is obtained by taking closures several times of $Y$.
Then 
\[
H_1(\mathbf{B}_n(X))=\begin{cases}
H_1(X)\oplus\mathbb{Z}& X\text{ is planar};\\
H_1(X)\oplus\mathbb{Z}_2& otherwise.
\end{cases}
\]
\end{cor}
\begin{proof}
By Lemma~\ref{lem:exactseq1}, the map $\coker(i_{Y,n})\to\coker(i_{X,n})$ is surjective, and Corollary~\ref{cor:elementary} implies that $\coker(i_{Y,n})$ is either $\mathbb{Z}$ or $\mathbb{Z}_2$.

If $Y$ is nonplanar, then $X$ is nonplanar and $\coker(i_{Y,n})\simeq\mathbb{Z}_2$. Hence so is $\coker(i_{X,n})$ by above.

Suppose $Y$ is planar but $X$ is nonplanar. Then since $Y$ is a surface, $X$ contains a nonplanar graph as mentioned in Remark~\ref{rmk:nonplanar}. Hence $\coker(i_{X,n})\simeq\mathbb{Z}_2$ by Corollary~\ref{cor:nonplanargraph}.

Suppose $X$ is planar, and consider the map $\iota_*:\coker(i_{X,n})\to \coker(i_{\mathbb{R}^2,n})\simeq\mathbb{Z}$ induced by the embedding $\iota:X\to\mathbb{R}^2$. Since $\iota_*$ is nontrivial and $\coker(i_{X,n})$ is generated by a single element, $\iota_*$ must be an isomorphism.
\end{proof}

\begin{lem}\label{lem:exactseq2}
Let $X, Y$ and $v,w$ be as given in Theorem~\ref{thm:connectedsum}. Then there exists an exact sequence as follows.
\[
\xymatrix@C=1pc{
1\ar[r]&H_1(\Theta_k)\ar[r]\ar[d]^{i_{\Theta_k,n}}&
H_1(X)\oplus H_1(Y)\ar[r]\ar[d]^{i_{X,n}\oplus i_{Y,n}}&
H_1(X\# Y)\ar[r]\ar[d]^{i_{X\# Y,n}}&1\\
&H_1(\mathbf{B}_n(\Theta_k))\ar[r]&
H_1(\mathbf{B}_n(X))\oplus H_1(\mathbf{B}_n(Y))\ar[r]&
H_1(\mathbf{B}_n(X\# Y))\ar[r]&1.
}
\]
\end{lem}

\subsection{Connectivities}
We adopt another notion for a decomposition, called a {\em cut}. This notion originated in graphs and is extended to complexes in slightly different ways as follows.

\begin{defn}
Let $X$ be a sufficiently subdivided simple complex. A set $\mathbf{v}=\{v_1,\dots,v_k\}$ of vertices in $\br(X)$ is a {\em $k$-cut} if $X_\mathbf{v}=X\setminus \st(\mathbf{v})$ is disconnected.

We say that a $k$-cut $\mathbf{v}$ is {\em trivial} if there exists a subcomplex $Y\subset X$ with $\mathbf{v}\subset\partial Y$ and $X$ is homeomorphic to $\widehat{(Y,\mathbf{v})}$,
%$X_\mathbf{v}$ has two connected components where one of them is homeomorphic to $T_k$ with $\partial T_k=\mathbf{v}$, 
and that $X$ is {\em vertex-$k$-connected} unless there is a nontrivial $(k-1)$-cut.
\end{defn}

\begin{rmk}
If $\mathbf{v}=\{v_1,\dots,v_k\}$ is a trivial $k$-cut and $\widehat Y=X$, then $\val(v_i)=1$ in $Y$ since $\mathbf{v}\subset\partial Y$.
Therefore $\val(v_i)=2$ in $X$ for all $i$.
However, since all $v_i$ are in $\br(X)$ and $X$ is simple, $\lk(v_i)=D^m\sqcup \{*\}$ for some $m\ge 1$.
Consequently, a trivial cut may exist only when $X$ is of dimension at least 2.

Especially, if there is a trivial 1-cut $v$, then it satisfies the assumption of Proposition~\ref{prop:edgecontraction} and therefore we may assume that there is no trivial 1-cut without loss of any generality.
\end{rmk}

%Every 1-cut is nontrivial since $\partial T_1$ consists of 2 points and vertex-1-connectedness is equivalent to (path-)connectedness.

For example, any nontrivial tree has a nontrivial 1-cut and so it is not vertex-2-connected, and $\Theta_k$ for $k\ge 3$ has a unique nontrivial 2-cut but no 1-cut, hence it is vertex-2-connected. %Moreover, for a graph $\Gamma$, every 1-cut is nontrivial and so a trivial 1-cut in $X$ may exists only when $X$ is of dimension at least 2.

%Notice that the complex $S_0$ has a trivial 1-cut.
%However, $S_0\equiv_B S_0'$ where $S_0'$ looks like the union of a plane and a half disk as depicted in Figure~\ref{fig:S_0'}, and is vertex-$k$-connected for any $k\ge 1$.

%In general, suppose that a simple complex $X$ has a trivial 1-cut $v$. Then $\lk(v)=\Gamma\sqcup\{w\}$ by the remark above.
%%since $v\in \br(X)$ by definition of a cut.
%Unless $\Gamma=S^1$, it satisfies the condition of Proposition~\ref{prop:graphedgecontraction}. Therefore we may contract the whole ray attached at $v$ to make $\val(v)=1$, and it reduces the number of trivial 1-cuts by 1.
%
%However, if $\Gamma=S^1$, then $\st(v)\simeq S_0$ and so by replacing $\st(v)$ with $S_0'$, we obtain $X'$, where $X\simeq_B X'$ and the number of trivial 1-cuts is reduced by 1 as well.
%Therefore without loss of any generality, we may assume that $X$ has no trivial 1-cut of valency 2. 
%In other words, if $X$ has a trivial 1-cut $v\in\br(X)$, then $\val(v)\ge 3$ and therefore $\st(v)$ is homeomorphic to $T_k$ for some $k\ge 3$ by simplicity of $X$.

\begin{figure}[ht]
\[
S_0'=\vcenter{\hbox{\includegraphics{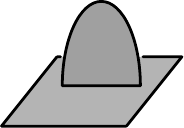}}}\quad\equiv_B\quad
\vcenter{\hbox{\includegraphics{S_0.pdf}}}=S_0
\]
\caption{A complex $S_0'$ which is braid equivalent to $S_0$}
\label{fig:S_0'}
\end{figure}

We introduce the famous result of Menger about the relationship between vertex-$k$-connectivity and the existence of embedded $\Theta_k$ as follows.
\begin{lem}\cite{M}
Let $\Gamma$ be a graph without a vertex of valency 1.
Then $\Gamma$ is vertex-$k$-connected if and only if for any $v,w\in\br(\Gamma)$, there is an embedding $(\Theta_k, \{0,0'\})\to(\Gamma,\{v,w\})$ of pairs.
\end{lem}

\begin{ex}[Vertex-3-connectivity for the union of two trees]\label{ex:twotrees2}
Recall the graph $\widehat T\#\widehat T'$ defined in Example~\ref{ex:twotrees}.
Then it is vertex-3-connected only if there is an embedding $(\Theta_3,\{0,0'\})\to(\widehat T\#\widehat T', \{v,w\})$ for any $v\in \br(T)$ and $w\in\br(T')$. Indeed, we may assume that such $\Theta_k$ always passes the point $1\in\partial T$.

Let $T_3^L$ and $T_3^R$ be two halves of $\Theta_k$ as before.
Then the restrictions of an embedding $\Theta_k\to\widehat T\#\widehat T$ to $T_3^L$ and $T_3^R$ give us a pair of equivalence classes $[s_{i,j}]$ and $[s_{i,j}']$ with respect to $\sim_T$ and $\sim_{T'}$, 
which are related as 
$s_{i,j}' =t_i^{-1}t_j^{-1}s_{i,j}^{-1}t_i t_j$ as described in Example~\ref{ex:tree} and Example~\ref{ex:twotrees}.

Therefore the vertex-3-connectivity of $\widehat T\#\widehat T'$ implies that any pair of generators for $\mathbf{B}_2(T)$ and $\mathbf{B}_2(T')$ are related, and so $\mathbf{B}_2(\widehat T\#\widehat T')$ is generated by $t_i$'s and only one $s_{i,j}$.
\end{ex}

\subsubsection{1-cuts}
Assume that $X$ has a 1-cut $v$ of valency $k\ge 2$, and $X_1\dots, X_m$ are connected components of $X_v$. Note that if $k=2$, then $m$ must be 2 and this is a 1-connected sum decomposition of $X$. Therefore we assume that $k\ge 3$.

For $1\le i\le k$, let $k_i$ be the number of vertices in $X_i$ which are adjacent to $v$.
Let $\mathbf{k}=(k_1,\dots,k_m)$, then $\sum_{i=1}^m k_i = k$.

Now we decompose $X$ into $(m+1)$ pieces, namely, $\widehat X_1,\dots, \widehat X_m$ and $\widehat T_{\mathbf{k}}$, via $k_i$-connected sums for all $1\le i\le m$, where $\widehat T_{\mathbf{k}}$ looks like a graph depicted in Figure~\ref{fig:1cut}. 
We apply Lemma~\ref{lem:exactseq2} for $X= \widehat T_{\mathbf{k}} \# \widehat X_1\#\dots\# \widehat X_m$ as follows.
\begin{align*}
\bigoplus_{i=1}^m H_1(\mathbf{B}_n(\Theta_{k_i}))
\xrightarrow{\oplus\tilde\Psi_i}
H_1(\mathbf{B}_n(\widehat T_{\mathbf{k}}))\oplus\bigoplus_{i=1}^m H_1(\mathbf{B}_n(\widehat X_i))
\longrightarrow
H_1(\mathbf{B}_n(X))
\longrightarrow 1.
\end{align*}

\begin{figure}[ht]
\[\vcenter{\hbox{\scriptsize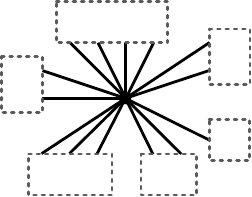}}\quad=\quad
\vcenter{\hbox{\scriptsize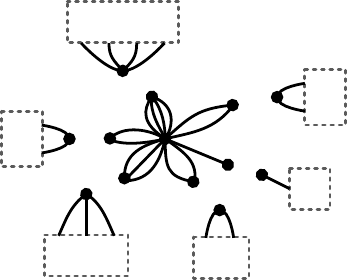}}\]
\caption{A decomposition of $X$ near 1-cut $v$ and a graph $\widehat T_{\mathbf{k}}$ with $\mathbf{k}=(3,2,1,2,4,2)$}
\label{fig:1cut}
\end{figure}

\begin{lem}\cite[Lemma~3.11]{KP}\label{lem:1cut}
The first homology group $H_1(\mathbf{B}_n(\widehat T_{\mathbf{k}}))$ is isomorphic to
\[H_1(\mathbf{B}_n(\widehat T_{\mathbf{k}})) = \mathbb{Z}^{r(n, k,m)} \oplus \bigoplus_{i=1}^m H_1(\mathbf{B}_n(\Theta_{k_i})).\]
\end{lem}

Hence the obvious embedding $\Theta_{k_i}\to \widehat{T}_{\mathbf{k}}$ yields an injection 
\[\tilde\Psi_i:H_1(\mathbf{B}_n(\Theta_{k_i}))\to H_1(\mathbf{B}_n(\widehat T_{\mathbf{k}})),\]
and therefore the sequence above becomes a short exact sequence. Moreover, we have the following lemma which is obvious by the decomposition in Lemma~\ref{lem:1cut}.

\begin{lem}
Let $X, v$ and $X_i$'s be as before. Then
\[H_1(\mathbf{B}_n(X))= \mathbb{Z}^{r(n,k,m)}\oplus \bigoplus_{i=1}^m H_1(\mathbf{B}_n(\widehat X_i)).\]
\end{lem}

In summary, we can say that each nontrivial 1-cut $v$ contributes to the first Betti number as much as $r(n,k,m)$ where $k=\val(v)$ and $m=\#(\pi_0(X_v))$. 

From now on, we assume that $X$ has no 1-cut. 
\begin{lem}\cite[Lemma~3.12]{KP}\label{lem:biconn}
Let $\Gamma$ be a graph without a 1-cut. Then for all $n\ge 2$, $H_1(\mathbf{B}_n(\Gamma))\simeq H_1(\mathbf{B}_2(\Gamma))$.

Therefore, $H_1(\mathbf{B}_n(\Theta_k))\simeq H_1(\Theta_k)\oplus\mathbb{Z}^{\binom{k-1}2}$.
\end{lem}

\subsubsection{2-cuts}
Assume that $X$ has no 1-cut but a nontrivial 2-cut $\mathbf{v}=\{v_1,v_2\}$, and $X_1,\dots, X_m$ are the connected components of $X_\mathbf{v}$ as before. 
Let $k_{i,j}$ be the number of components of $\lk(v_j)$ in $X_i$ and $\mathbf{k}_j=(k_{1,j},\dots,k_{m,j})$ for $1\le i\le m, j=1,2$.

Similar to above, we decompose $X$ into $(m+1)$-pieces via $(k_{i,1}+k_{i,2})$-connected sums as depicted in Figure~\ref{fig:2cut}. The connected summands will be denoted by $\widehat X_1,\dots, \widehat X_m$ and $\Theta_{\mathbf{k}_1, \mathbf{k}_2}$.
Then by Lemma~\ref{lem:exactseq2}, $H_1(\mathbf{B}_n(X))$ is isomorphic to the cokernel of 
\[
\xymatrix{
\displaystyle{\bigoplus_{i=1}^m H_1(\mathbf{B}_n(\Theta_{k_{i,1}+k_{i,2}}))}\ar[r]^-{\oplus \tilde\Psi_i}&
H_1(\mathbf{B}_n(\Theta_{\mathbf{k}_1,\mathbf{k}_2}))\oplus\displaystyle{\bigoplus_{i=1}^m H_1(\mathbf{B}_n(\widehat X_i)).}
}
\]

\begin{figure}[ht]
\[X\quad=\quad\vcenter{\hbox{\scriptsize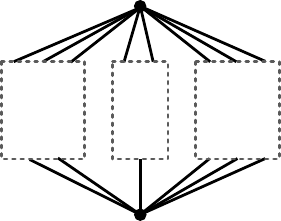}}\quad=\quad
\vcenter{\hbox{\scriptsize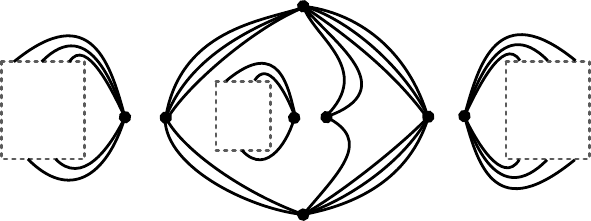}}\]
\caption{A decomposition of $X$ near a 2-cut $\mathbf{v}$ and a graph $\Theta_{\mathbf{k}_1,\mathbf{k}_2}$ with $\mathbf{k}_1=(3,2,3)$ and $\mathbf{k}_2=(2,1,3)$}
\label{fig:2cut}
\end{figure}

\begin{figure}[ht]
\[
\vcenter{\hbox{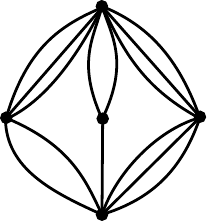}}\quad\stackrel{q}{\longleftarrow}\quad
\vcenter{\hbox{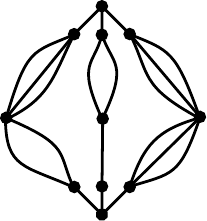}}\quad=\quad
\vcenter{\hbox{\scriptsize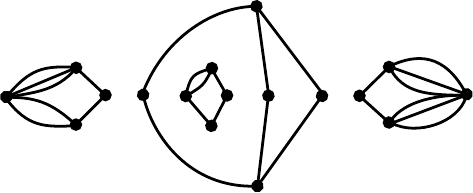}}
\]
\caption{A decomposition of $\widetilde\Theta_{\mathbf{k}_1,\mathbf{k}_2}$ via 2-connected-sums}
\label{fig:tildethetadecomposition}
\end{figure}

Let $\Theta_{a,b,c}$ denote a (possibly subdivided) graph obtained by replacing respective edges of the triangle with $a$, $b$ and $c$ multiple edges.
We take $2$-connected sums between $\Theta_{k_{i,1},k_{i,2},1}$ and $\Theta_m$ to obtain $\widetilde\Theta_{\mathbf{k}_1,\mathbf{k}_2}$. Then $\Theta_{\mathbf{k}_1,\mathbf{k}_2}$ comes from $\widetilde\Theta_{\mathbf{k}_1,\mathbf{k}_2}$ by contracting all edges adjacent to vertices of $\Theta_m$. See Figure~\ref{fig:tildethetadecomposition}.

We want to use $\widetilde\Theta_{\mathbf{k}_1,\mathbf{k}_2}$ instead of $\Theta_{\mathbf{k}_1,\mathbf{k}_2}$.
That is, we define $\widetilde X$ by taking $(k_{i,1}+k_{i,2})$-connected-sums between $\widehat X_i$ and $\widetilde\Theta_{\mathbf{k}_1,\mathbf{k}_2}$.
The lemma below ensures that we can safely do this.

\begin{lem}\cite[Lemma~3.14]{KP}\label{lem:2cut}
Let $q:\widetilde\Theta_{\mathbf{k}_1,\mathbf{k}_2}\to\Theta_{\mathbf{k}_1,\mathbf{k}_2}$ be the quotient map and $q^*:\mathbf{B}_n(\Theta_{\mathbf{k}_1,\mathbf{k}_2})\to\mathbf{B}_n(\widetilde\Theta_{\mathbf{k}_1,\mathbf{k}_2})$ be the map defined in Proposition~\ref{prop:generaledgecontraction}.

Then $q_*$ induces an isomorphism between abelianizations, that is, the first homology group $H_1(\mathbf{B}_n(\Theta_{\mathbf{k}_1,\mathbf{k}_2}))$ is isomorphic to
$H_1(\mathbf{B}_n(\widetilde \Theta_{\mathbf{k}_1,\mathbf{k}_2}))$
\end{lem}

Therefore $H_1(\mathbf{B}_n(X))=H_1(\mathbf{B}_n(\widetilde X))$ and now we can decompose $\widetilde X$ by using $2$-connected-sums into $\widetilde X_i$'s and $\Theta_m$, where $\widetilde X_i=\widehat X_i \# \Theta_{k_{i,1},k_{i,2},1}$.

\begin{figure}[ht]
\begin{align*}
\widetilde X\quad=\quad&\vcenter{\hbox{\scriptsize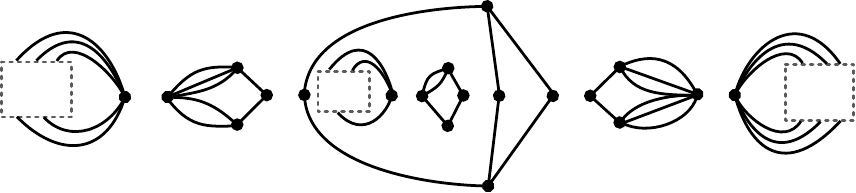}}\\
\quad=\quad&\vcenter{\hbox{\scriptsize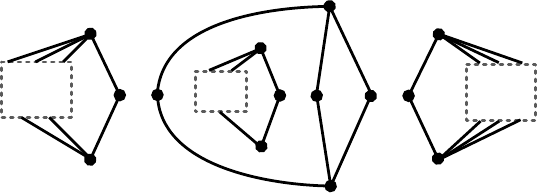}}\\
\quad=\quad&\Theta_3\#\widetilde X_1\#\widetilde X_2\#\widetilde X_3.
\end{align*}
\caption{A 2-connected-sum decomposition of $\widetilde X$ near a 2-cut}
\label{fig:tildeX}
\end{figure}

By Lemma~\ref{lem:exactseq2} again, we have a short exact sequence
\[1\to\bigoplus_{i=1}^m H_1(\mathbf{B}_n(\Theta_2))\to
H_1(\mathbf{B}_n(\Theta_m))\oplus\bigoplus_{i=1}^m H_1(\mathbf{B}_n(\widetilde X_i))\to
H_1(\mathbf{B}_n(\widetilde X))\to 1.\]

\begin{lem}
Let $X, \mathbf{v}$ and $X_i$'s be as above. Then
\begin{align*}
H_1(\mathbf{B}_n(X))\oplus\mathbb{Z}^m&=H_1(\mathbf{B}_n(\Theta_m))\oplus\bigoplus_{i=1}^m H_1(\mathbf{B}_n(\widetilde X_i))\\
&=\mathbb{Z}^{\binom{m}2}\oplus\bigoplus_{i=1}^m H_1(\mathbf{B}_n(\widetilde X_i)).
\end{align*}
\end{lem}
\begin{proof}
This follows easily from the above exact sequence and $H_1(\mathbf{B}_n(\Theta_2))=\mathbb{Z}$.
\end{proof}

\subsubsection{Vertex-3-connected complexes}

We claim the following.
\begin{prop}\label{prop:3conn}
Let $X$ be a simple and vertex-3-connected complex. Then for $n\ge 2$,
\[
H_1(\mathbf{B}_n(X))=\begin{cases}
H_1(X)\oplus\langle[\sigma]\rangle&X\text{ is planar};\\
H_1(X)\oplus\langle[\sigma]\rangle/\langle 2[\sigma]\rangle&X\text{ is nonplanar}.
\end{cases}
\]
\end{prop}

This is a generalization of the result for vertex-3-connected graphs stated in \cite[Lemma~3.15]{KP}, and we can prove Theorem~\ref{thm:emb}~(3) with the aid of this proposition as follows.

\begin{proof}[Proof of Theorem~\ref{thm:emb}~(3)]
Since vertex-1 and 2-connected sums and 1-cut and 2-cut decompositions preserve planarity, $X$ is nonplanar if and only if $X$ has a nonplanar vertex-3-connected component with respect to 1-cut and 2-cut decompositions.

On the other hand, Lemma~\ref{lem:1cut} implies that $H_1(\mathbf{B}_n(X))$ has torsion if and only if one of its vertex-2-connected components does. However, Lemma~\ref{lem:2cut} does not imply directly the corresponding result because of $\mathbb{Z}$ summands.
Each $\mathbb{Z}$ summand is mapped to a summand of the homology group $H_1(\widetilde X_i)$ of a vertex-3-connected component $\widetilde X_i$ via the map induced from the embedding $\Theta_2=S^1\to \widetilde X_i$.
Hence $H_1(\mathbf{B}_n(X))$ has torsion for a vertex-2-connected complex $X$ if and only if one of its vertex-3-connected components does.

Finally, Proposition~\ref{prop:3conn} completes the proof.
\end{proof}

For the rest of the paper, we will prove Proposition~\ref{prop:3conn}. Hence from now on, we suppose that $X$ is a simple and vertex-3-connected complex. 

\begin{lem}\label{lem:3conntrees}
Let $T$ and $T'$ be trees with $k=\#(\partial T)=\#(\partial T')$. Suppose $\widehat T\#\widehat T'$ is vertex-3-connected.
Then Proposition~\ref{prop:3conn} holds for $\widehat T\#\widehat T'$.
\end{lem}
\begin{proof}
This is a special case of Lemma~3.15 in \cite{KP}, and we introduce a new proof.

By Lemma~\ref{lem:biconn}, it suffices to consider $H_1(\mathbf{B}_2(\widehat T\#\widehat T'))$.
Recall the group presentation for $\mathbf{B}_2(\widehat T\#\widehat T')$ from Example~\ref{ex:twotrees}. Then its abelianization has a presentation as follows.
\[
H_1(\mathbf{B}_2(\widehat T\#\widehat T'))=\mathbb{Z}^{k-1}\oplus\bigoplus_{2\le i<j\le k}\langle [s_{i,j}]\rangle\bigg/ 
\left\langle [s_{i,j}]-[s_{i',j'}] \left| \begin{matrix}(i,j)\sim_T(i',j')\text{ or}\\(i,j)\sim_{T'}(i',j')\end{matrix}\right.\right\rangle.
\]

Then as shown in Example~\ref{ex:twotrees2}, vertex-3-connectedness implies that the two equivalence relations $\sim_T$ and $\sim_{T'}$ are engaged with each other so that they make all $[s_{i,j}]$'s equivalent. 
%One can see this by interpreting equivalence relations as partitions on $\{1,\dots,k\}$. Then the existence of inequivalent classes implies the existence of a nontrivial 2-cut.
Hence the cokernel of $i_{\widehat T\#\widehat T',2}$ is generated by a single element $[s]$. 

If $\widehat T\#\widehat T'$ is nonplanar, then they must be glued in a twisted way. That is, one of $[s_{i,j}]$'s must be identified with its inverse $-[s_{i,j}]$, and therefore $s$ is 2-torsion since 
\[2[s]=[s_{i,j}]-(-[s_{i,j}])=0.\]
\end{proof}

\begin{prop}\label{prop:connectedsum}
Let $X=Y\# Z$. Suppose that Proposition~\ref{prop:3conn} holds for all vertex-3-connected subcomplexes of $Y$ and $Z$. Then it does for $X$ as well.
\end{prop}
\begin{proof}
Let $X=(Y,\vec v)\# (Z,\vec w)$ for $\lk(\vec v) =(v_1,\dots,v_k)$ and $\lk(\vec w)=(w_1,\dots,w_k)$. Then both $Y_v$ and $Z_z$ are connected by definition, and therefore both $Y$ and $Z$ are vertex-2-connected.

However, in general, $Y$ and $Z$ are not necessarily vertex-3-connected, and if a 2-cut exists in $Y$ or $Z$, then one of the two vertices is precisely $v$ or $w$, respectively.
Then from the 2-cut decompositions for $Y$ and $Z$, we can obtain two trees $T_Y\subset Y_v$ and $T_Z\subset Z_w$ such that $\partial T_Y=\lk(v)$ and $\partial T_Z=\lk(w)$, where vertices in $T_Y$ and $T_Z$ correspond to vertex-3-connected components in $Y$ and $Z$, respectively.
Since $X$ is vertex-3-connected, so is $\widehat T_Y\#\widehat T_Z$ by construction.

Moreover, by the assumption about $Y$ and $Z$, $\coker(i_{Y,n})$ and $\coker(i_{Z,n})$ are generated by $r_2(T_Y)$ and $r_2(T_Z)$ elements, respectively.

The commutative diagram in Lemma~\ref{lem:exactseq2} produces an exact sequence
\[\coker(i_{\Theta_k,n})\to\coker(i_{Y,n})\oplus\coker(i_{Z,n})\to\coker(i_{X,n})\to 1.\]

However, the quotient of $\coker(i_{Y,n})\oplus\coker(i_{Z,n})$ by the image of $\coker(i_{\Theta_k,n})$ is nothing but $\coker(i_{\widehat T_Y\#\widehat T_Z,n})$ and by Lemma~\ref{lem:3conntrees}, it is either $\mathbb{Z}$ or $\mathbb{Z}_2$.

Finally, it is $\mathbb{Z}_2$ if and only if either one of vertex-3-connected components of $Y$ and $Z$ is nonplanar, or both $Y$ and $Z$ are planar but $\widehat T_Y\#\widehat T_Z$ is nonplanar. Since these conditions are equivalent to the nonplanarity of $X$, we are done.
\end{proof}

Therefore we may assume furthermore that $X$ is not decomposable in a nontrivial way via $k$-connected sum for all $k\ge 1$. However, note that $X$ might be expressible as a closure.

Indeed, there is no such complex of dimension 1, a graph, as follows. If it exists, then it has at least 3 vertices of valency $\ge 3$ by vertex-3-connectedness, but at most 3 as well since it is always decomposable when a graph has at least 4 vertices of valency $\ge 3$ as follows. First we divide the set $V(\Gamma)$ of vertices of valency $\ge3$ into two parts $V_1$ and $V_2$, where both the full subgraphs $\Gamma_1$ and $\Gamma_2$ containing $V_1$ and $V_2$ are connected.
Then for each $i$, consider the closure of the complement of $\st(\Gamma_i)$ along its boundary. Then $\Gamma$ is nothing but the connected sum of these two complexes. Hence it is decomposable, and therefore the only possibilities are $\Theta_{a,b,c}$ defined above. Since vertex-3-connectedness implies the absence of multiple edges, two of $a,b$ and $c$ must be 1. This is a contradiction.

Therefore $X$ must be obtained by taking closures several times of an elementary complex of dimension 2. 
However, this case has been treated already in Corollary~\ref{cor:closures}. This completes the proof of Proposition~\ref{prop:3conn}.

\appendix
\section{More deformations}
All the attaching cells we consider previously induce braid equivalences. However we can consider more deformations which do not induce braid equivalences, but are helpful in computing braid groups. 
In this section, we introduce several 2-dimensional deformations $\iota:X\to Y$ such that $\iota$ induces a surjection between braid groups, and try to characterize their kernel. Hence we may use these deformations to compute $\mathbf{B}_n(Y)$ from $\mathbf{B}_n(X)$ which is already known.

For example, by using the deformations we will introduce below, one can deform a trivalent, vertex-3-connected graph into a surface by keeping the cokernel of $i_{(\cdot),n}:H_1(\cdot)\to H_1(\mathbf{B}_n(\cdot))$ unchanged. Therefore one may deduce another proof for Proposition~\ref{prop:3conn}.

\begin{defn}
An embedding $\bar\iota:X\to Y$ is a {\em local deformation of type $\iota:U\to V$ via $i:U\to X$} if $i(U)=\overline{\st(K)}$ for some $K\subset X$ and $Y$ is the push-out of $\iota$ and $i$.
\[
\xymatrix{
X\ar[r]^{\bar\iota}&Y\\
U\ar[u]^i\ar[r]_\iota& V\ar[u]_{\bar i}
}
\]
\end{defn}

\subsection{Attaching a disk at a trivalent vertex}
Let $v$ be a trivalent vertex in $X$ which we attach a disk on, and let $\bar D$ be the result as depicted in Figure~\ref{fig:attachingdisk}.

Suppose $v$ is not a 1-cut of $X$. Then attaching a disk is nothing but $3$-connected sum with $\widehat D_3$, where $\widehat D_3$ is a closure of a disk along 3 points in the boundary. So Theorem~\ref{thm:connectedsum} and Lemma~\ref{lem:exactseq2} are applicable. We denote $X\# \widehat D_3$ by $\overline X$. %Note that this modification always preserves planarity.

\begin{figure}[ht]
\[T_3=\vcenter{\hbox{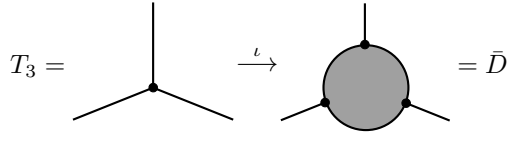}}=\bar D\]
\caption{Attaching a disk on a trivalent vertex}
\label{fig:attachingdisk}
\end{figure}

\begin{prop}
If $v$ is not a 1-cut of $X$, then the 2-braid groups $\mathbf{B}_2(X)$ and $\mathbf{B}_2(\overline X)$ are isomorphic via $\iota_*$.
\end{prop}
\begin{proof}
By Corollary~\ref{cor:boundary}, $\mathbf{B}_2(\widehat D_3)\simeq\mathbf{B}_2(D^2)\ast\langle t_2, t_3\rangle$, and $\mathbf{B}_2(D^2)\simeq\mathbf{B}_2(T_3)\simeq\mathbb{Z}$. Hence the assertion follows from Theorem~\ref{thm:connectedsum}.
\end{proof}

In general, we consider the local deformation $\bar\iota:X\to Y$ of type $\iota:T_3\to \bar D$ via $i:T_3\to X$, and let $\bar\iota_*:\mathbf{B}_n(X)\to\mathbf{B}_n(Y)$ be the induced homomorphism from $\bar\iota$. 

\begin{prop}
The map $\bar\iota_*$ is surjective and its kernel $\ker \bar\iota_*$ is generated by $i_*(\ker\iota_*)$.
\end{prop}
\begin{proof}
We first regard $T_3$ and $\bar D$ as subspaces of $\mathbb{R}^2$ as follows.
\begin{align*}
T_3&=\{(x,0)||x|\le2\} \cup \{(0,y)|-1\le y\le 0\},\\
D^2&=\{(x,y)||x|\le 1, 0\le y\le 1\},\\ 
\bar D &= T_3\cup D.
\end{align*}

\begin{figure}[ht]
\[
\xymatrix@C=3pc{
T_3=\vcenter{\hbox{\includegraphics{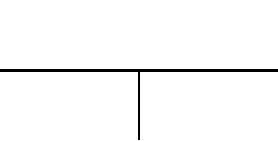}}}\ar@<-.3pc>[r]_{\iota}\qquad&\qquad
\vcenter{\hbox{\includegraphics{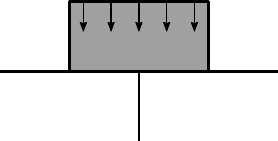}}}=\bar D\ar@<-.3pc>[l]_r
}
\]
\caption{A complex $\bar D$ and the projection $r$}
\label{fig:barD}
\end{figure}

We consider a projection (or strong deformation retract) $r$ of $\bar D$ onto $T_3$ that projects $D^2$ to $x$-axis, and it extends to the projection of $Y$ onto $X$. Then similar to the proof of Proposition~\ref{prop:highercell}, $r$ does not induce a map between configuration spaces, and let $B_n^{r\text{-fail}}$ be the subspace of $B_n(Y)$ consisting of configurations $\mathbf{x}$ that $r$ can not act on.

Since $B_n^{r\text{-fail}}$ is of codimension 1, a path $\gamma(t)$ in $B_n(Y)$ in general position intersects this space finitely many times.
Indeed, it happens only if exactly two points in $\gamma(t)\in B_n(Y)$ are lying in the ray $\{(x_0,y)|0\le y\le 1\}$ for some $x_0$. However, it can be homotoped to $\gamma'$ so that $x_0=0$, and homotoped into $B_n(X)$ by using $(-y)$-axis as desired.

Let $f=\{f_1,\dots, f_n\}:(D^2,\partial D^2)\to(B_n(Y),B_n(X))$ be a homotopy disk in general position with respect to $B_n^{r\text{-fail}}$. Then a failure locus $F=f^{-1}(B_n^{r\text{-fail}})$ is a 1-dimensional subcomplex of $D^2$ away from $\partial D^2$. Indeed, $F$ is a union of arcs (including circles) which intersect {\em pairwise transversely}, and is not necessarily closed.

Suppose that we travel along a path $\gamma(t)$ in $D^2$ which passes through one of arcs of $F$ at $t_0$.
Then there are $i$ and $j$ such that $f_i(\gamma(t_0))$ and $f_j(\gamma(t_0))$ make a failure of $r$.
Moreover, the sign of the difference of $x$-coordinates of these two paths must change near $t_0$ since $f$ is in general position. 
That is, two paths $f_i(\gamma(t))$ and $f_j(\gamma(t))$ really make a crossing as usual.

Note that there are two types of intersection points between arcs in $F$, that is, either two arcs or three arcs intersect at one point in $D^2$.

The former case happens when the failure occurs at two different rays simultaneously, therefore four points of a configuration are involved.
On the other hand, the latter case happens when the failure occurs at a single ray but three points are lying on it simultaneously. See Figure~\ref{fig:failurelocus}.

\begin{figure}[ht]
\[
f:\vcenter{\hbox{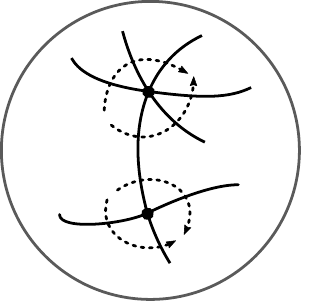}}\longrightarrow B_n(Y)
\]
\[
\vcenter{\hbox{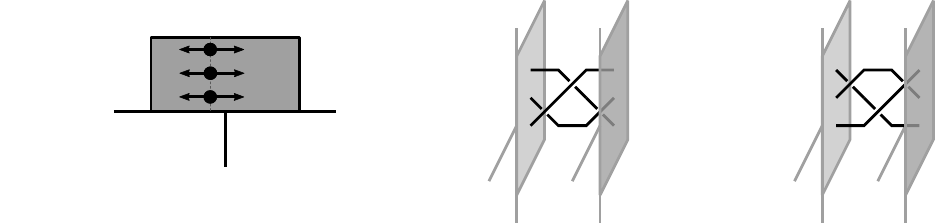}}
\]
\[
\vcenter{\hbox{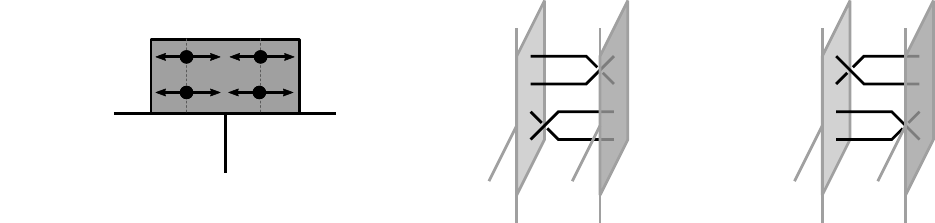}}
\]
\caption{A failure locus for the projection $r$ and curves near by intersections}
\label{fig:failurelocus}
\end{figure}

Evidently, the small loops $\delta$'s around these two kinds of intersections correspond to the braid relations such that $\sigma_i\sigma_j=\sigma_j\sigma_i$ for $|i-j|>1$ and $\sigma_i\sigma_{i+1}\sigma_i=\sigma_{i+1}\sigma_i\sigma_{i+1}$, which can be considered as images of $\ker(\iota_*)$.

Now we perturb $f$ near $\delta$ so that $f(\delta)$ belongs to $B_n(X)$ as before. Then $r(f(\delta))$ is lying in $\ker(\iota_*)$ by definition.
Perturb $f$ further so that all arcs without intersection point disappear in a failure locus.
Then the failure locus consists of those near intersection points, which implies that $f\in \ker(\bar\iota_*)$ is generated by $i_*(\ker(\iota_*))$ by choosing paths from the basepoint of $\partial D^2$ to a point in each $\delta_i$.
\end{proof}

\begin{rmk}
Let $\iota_{n*}:\mathbf{B}_n(T_3)\to\mathbf{B}_n(\bar D)$.
Since $\ker(\iota_{n*})$ is generated by two braid relations which can be considered as 3- or 4-braids, it is generated by the images of $\ker(\iota_{4*})$ and $\ker(\iota_{3*})$ under the maps
\[
\mathbf{B}_m(T_3)\to\mathbf{B}_n(T_3)\stackrel{i_*}{\longrightarrow}\mathbf{B}_n(\bar D),
\]
for $m=3$ or $4$. Therefore $\ker(\bar\iota_*)$ is also generated by these images of them under the compositions with $i_*$.
\end{rmk}

\subsection{Edge-to-band replacement}
Let $D_i$ be {\em oriented} disks, and $D_1\vee_\partial D_2=D_1\sqcup D_2/*_1\sim*_2$ be the boundary wedge sum of $D_i$'s where $*_i\in\partial D_i$. By Proposition~\ref{prop:edgecontraction}, this is braid equivalent to the 1-connected sum $D_1\# D_2$, joining two disks by an edge. For the sake of convenience, we use the wedge sum instead.
Then by Theorem~\ref{thm:connectedsum}, 
\[\mathbf{B}_2(D_1\vee_\partial D_2)=\mathbf{B}_2(D_1)\ast\mathbf{B}_2(D_2)=\langle\sigma_{1,1},\sigma_{1,2}\rangle,\]
where $\sigma_{1,i}$ is the generator for $\mathbf{B}_2(D_i)$ such that $[\sigma_{1,i}^2]\in H_1(D_i\setminus\{0\})=H_1(\partial D_i)\simeq H_2(D_i,\partial D_i)$ corresponds to the orientation of $D_i$.
\begin{figure}[ht]
\[
\vcenter{\hbox{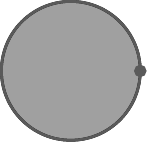}}\vee_\partial
\vcenter{\hbox{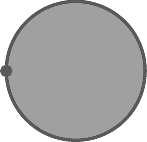}}=
\vcenter{\hbox{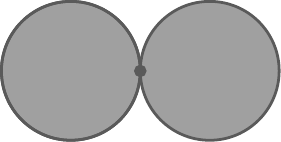}}
\]
\caption{A boundary wedge sum of two disks}
\label{fig:disks}
\end{figure}

Note that any orientation preserving embedding $\iota:D_1\vee_\partial D_2\to D^2$ induces an homomorphism $\iota_*$ between braid groups which maps both $\sigma_{1,i}$ to $\sigma_1$.
Therefore the kernel $\ker \iota_*$ is generated by $\sigma_{1,1}\sigma_{1,2}^{-1}$.
Moreover, even for $n\ge 3$, the kernel of $\mathbf{B}_n(D_1\vee_\partial D_2)\to\mathbf{B}_n(D^2)$ is generated by the image of $\ker \iota_*$ under the inclusion map $i_*:\mathbf{B}_2(D_1\vee_\partial D_2)\to\mathbf{B}_n(D_1\vee_\partial D_2)$ coming from Proposition~\ref{prop:injection}. This follows easily from the group presentations for $\mathbf{B}_n(D_1\vee_\partial D_2)$ and $\mathbf{B}_n(D^2)$.

As before, let us consider a local deformation $\bar\iota:X\to Y$ of type $\iota:D_1\vee_\partial D_2\to D^2$ via $i$.
Note that $Y$ can be regarded intuitively as a blowing-up of $X$ at $v$, and there are two possibilities for constructing $Y$ according to the choices of the orientations on $D_i$'s.

\begin{prop}\label{prop:attach2cell}
The kernel $\ker \bar\iota_*$ is generated by $i_*(\ker \iota_*)$.
\end{prop}
\begin{proof}
Let $a=\{v\}\times[0,1]\subset \overline X$ and let $r:Y\to Y/a=X$ be the quotient map. Then we define a failure $B_n^{r\text{-fail}}$ of $r$ as a subspace of $B_n(Y)$ consisting of configurations $x$ which intersect $a$ at least twice.
Then the codimension of $B_n^{r\text{-fail}}$ is 2 and therefore any path in $B_n(Y)$ can be homotoped into $B_n(X)$ as before.

For a disk $f$, the failure locus $F=f^{-1}(B_n^{r\text{-fail}})$ is a set $\{z_1,\dots, z_m\}$ of finite points in $D^2$, and for each $z\in F$, there exists $i$ and $j$ so that $\{f_i(z),f_j(z)\}\subset a$.
Now we consider two subspaces $F_i=f_i^{-1}(a)$ and $F_j=f_j^{-1}(a)$, whose intersection contains $z$. Then a small loop $\delta$ around $z$ can be separated into four pieces as depicted in Figure~\ref{fig:homotopydisk}.

\begin{figure}[ht]
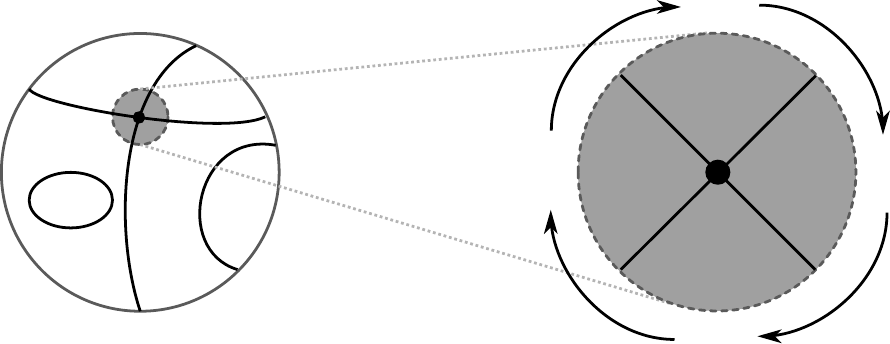
\begin{align*}
r(f(\delta_1))=\vcenter{\hbox{\includegraphics[scale=0.6]{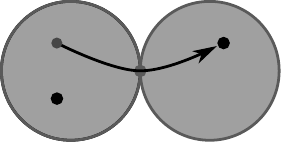}}}&\qquad
r(f(\delta_2))=\vcenter{\hbox{\includegraphics[scale=0.6]{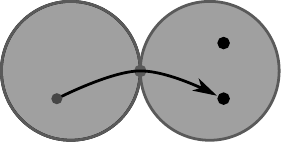}}}\\
r(f(\delta_3))=\vcenter{\hbox{\includegraphics[scale=0.6]{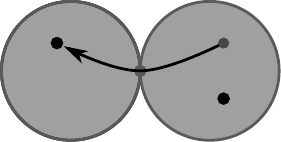}}}&\qquad
r(f(\delta_4))=\vcenter{\hbox{\includegraphics[scale=0.6]{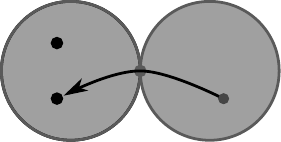}}}
\end{align*}
\caption{A null-homotopic disk and a loop $\delta$ near $z$ in a failure locus}
\label{fig:homotopydisk}
\end{figure}

Hence by choosing $\delta$ close enough to $z$, we may assume that all but only $f_i$ and $f_j$ remains stationary. However, the element $r_*(\{f_i(\delta),f_j(\delta(t))\})$ in $\mathbf{B}_2(D_1\vee_\partial D_2)$ is a generator of $\ker(\iota_*)$. In other words, any $z\in F$ corresponds to an element in the image of $\ker(\iota_*)$ under 
\[
\mathbf{B}_2(D_1\vee_\partial D_2)\to\mathbf{B}_n(D_1\vee_\partial D_2)\stackrel{i_*}{\longrightarrow}\mathbf{B}_n(X).
\]
\end{proof}

\subsection{2-dimensional capping-off}
Let $\gamma$ be an embedded loop in $\partial X$, which is not homotopically trivial and let $Y=X\sqcup_\gamma D^2$ be a capping-off of $X$ along $\gamma$. 
Then $X\to Y$ can never be a braid equivalence since their fundamental groups are different.

Since $\gamma\subset\partial X$, $\gamma\cap\br(X)=\emptyset$ and therefore $\overline{\st(\gamma)}$ is a manifold with boundary containing an annulus $A$, where one of whose boundary is $\gamma$.
We parametrize $A$ as 
\[A=\{(x,y)\in\mathbb{R}^2|1/2\le x^2+y^2\le 1\},\]
so that the inner boundary $\partial_{1/2}$ is precisely $\gamma$, and let $D^2$ be the unit disc in $\mathbb{R}^2$. 

Then we have the following commutative diagram
\[
\xymatrix{
X\ar[r]^-{\equiv_B}&Y\setminus\{0\}\ar[r]^{\bar\iota}&Y\\
A\ar[u]^i\ar[r]^-{\equiv_B}& D^2\setminus\{0\}\ar[u]^i\ar[r]_{\iota}&D^2\ar[u]_{\bar i}
}
\]
which satisfies $X\setminus i(A)\simeq (Y\setminus\{0\})\setminus i(D^2\setminus\{0\}) \simeq Y\setminus \bar i(D^2)$.
The braid equivalences in the left come from Proposition~\ref{prop:highercell} and Corollary~\ref{cor:2cell}.

\begin{prop}
The kernel $\ker(\bar\iota_*)$ of $\bar\iota_*$ is generated by the image of $\gamma\in\pi_1(A)\subset \mathbf{B}_n(A)$ under $i_*$.
\end{prop}
\begin{proof}
We consider a subspace $B_{n-1;1}(Y\setminus\{0\}; \{0\})=B_{n-1}(Y\setminus\{0\})\times\{0\}$ of $B_n(Y)$. Then it is obviously of codimsion 2.
Hence we may assume that any path can be homotoped to a path avoiding $0$ and any homotopy disk intersects $\{0\}$ finitely many times.

Moreover, for a homotopy $f$ and $z\in D^2$ with $0\in f(z)$, the image $f(\delta(t))$ of a small enough loop $\delta(t)$ enclosing $z$ via $f$ is lying in $B_n(Y\setminus\{0\})$ and we may assume that all but 1 point, say $f_1(\delta(t))$, remains stationary since $\delta(t)$ can be arbitrarily small.

On the contrary, the image of $f_1(\delta(t))$ is homotopic to $\gamma$ or its inverse. This completes the proof.
\end{proof}


\begin{thebibliography}{aa00}
	\bibitem{Ab} A.~Abrams, {\em Configuration space of braid groups of graphs}, Ph.D. thesis in UC Berkeley,
ProQuest LLC, Ann Arbor, MI, 2000.
	\bibitem{Art} E. Artin, {\em Theorie der Z\"opfe}, Hamburg Abh. 4 (1925), 47--72.
	\bibitem{KKP} J.~Kim, K.~Ko, H.~Park, {\em Graph braid groups and right-angled Artin groups}, Trans. Amer. Math. Soc. 364 (2012), no. 1, 309--360.
	\bibitem{KP} K.~Ko, H.~Park, {\em Characteristics of graph braid groups}, Discrete Comput. Geom. 48 (2012), no. 4, 915--963.
	\bibitem{Bel} P.~Bellingeri, {\em On presentations of surface braid groups}, J. Algebra 274 (2004), no. 2, 543--563. 
	\bibitem{Bir} J.~Birman, {\em Braids, links, and mapping class groups}, Annals of Mathematics Studies, No. 82. Princeton University Press, Princeton, N.J.; University of Tokyo Press, Tokyo, 1974. ix+228 pp.
	\bibitem{Cor} J.~Corson, {\em Complexes of groups}, Proc. London Math. Soc. (3) 65 (1992), no. 1, 511--519.
	\bibitem{For} R.~Forman, {\em Morse theory for cell complexes}, Adv. Math. 134 (1998) 90--145.
	\bibitem{FS} D.~Farley, L.~Sabalka, {\em Discrete Morse theory and graph braid groups}, Algebr. Geom. Topol. 5 (2005), 1075--1109.
	\bibitem{GG} D.~L.Gon\c calves, J.~Guaschi, {\em The braid groups of the projective plane}, Algebr. Geom. Topol. 4 (2004), 757--780.
	\bibitem{Gh} R.~Ghrist, {\em Configuration spaces and braid groups on graphs in robotics}, Knots, braids, and mapping class groups---papers dedicated to Joan S. Birman (New York, 1998), pp. 29--40, AMS/IP Stud. Adv. Math. 24, Amer. Math. Soc., Providence, RI, 2001.
	\bibitem{M} K.~Menger, {\em Zur allgemeinen Kurventheorie}, Fund. Math. 10 (1927), 96--115.
	\bibitem{Sa} L.~Sabalka, {\em On rigidity and the isomorphism problem for tree braid groups}, Groups Geom. Dyn. 3 (2009), no. 3, 469--523.
%	\bibitem {GGT} {\em Geometric group theory}, Proceedings of a Special Research Quarter held at the Ohio State University, Columbus, Ohio, 1992. Edited by Ruth Charney, Michael Davis and Michael Shapiro. Ohio State University Mathematical Research Institute Publications, 3. Walter de Gruyter \& Co., Berlin, 1995. x+186 pp.
\end{thebibliography}
\end{document}